\documentclass[a4paper,10pt]{article}

\usepackage{a4wide}
\usepackage{amsmath}
\usepackage{amsthm}
\usepackage{amssymb}
\usepackage{graphicx}
\usepackage[algo2e]{algorithm2e}

\numberwithin{figure}{section}
\numberwithin{table}{section}

\newtheorem{thm}{Theorem}[section]
\newtheorem{lem}[thm]{Lemma}
\newtheorem{prop}[thm]{Proposition}
\newtheorem{cor}[thm]{Corollary}

\theoremstyle{definition}
\newtheorem{defi}[thm]{Definition}
\newtheorem{problem}{Problem}

\theoremstyle{remark}
\newtheorem{rem}[thm]{Remark}

\newcommand{\norm}[1]{\lVert#1\rVert}

\newcommand{\rr}{\mathbb{R}}
\newcommand{\nn}{\mathbb{N}}
\newcommand{\test}{\varphi}
\renewcommand{\P}{\mathcal{P}}
\newcommand{\C}{\mathcal{C}}

\newcommand{\U}{U_{\text{ad}}}
\newcommand{\J}{\mathcal{J}}
\newcommand{\Y}{\mathcal{Y}}

\title{The Quasi-Neutral Limit in Optimal Semiconductor Design}
\author{Ren\'e Pinnau, Claudia Totzeck, Oliver Tse}
\date{}

\begin{document}
\maketitle
\begin{abstract}
We study the quasi-neutral limit in an optimal semiconductor design problem constrained by a nonlinear, nonlocal Poisson equation modelling the drift diffusion equations in thermal equilibrium. While a broad knowledge on the asymptotic links between the different models in the semiconductor model hierarchy exists, there are so far no results on the corresponding optimization problems available. Using a variational approach we end up with a bi-level optimization problem, which is thoroughly analysed. Further, we exploit the concept of $\Gamma$-convergence to perform the quasi-neutral limit for the minima and minimizers. This justifies the construction of fast optimization algorithms based on the zero space charge approximation of the drift-diffusion model. The analytical results are underlined by numerical experiments confirming the feasibility of our approach. 
 \\[0.5cm] 
\\[1ex]
\end{abstract}

\section{Introduction}
Nowadays, semiconductor devices play a crucial role in our society due to the increasing use of technical equipment in which  more and more functionality is combined. The ongoing miniaturization in combination with less energy consumption and increased efficiency requires that the designs cycles for the upcoming device generations are shortened significantly. Hence, black-box optimization approaches like genetic algorithms or derivative-free optimization are not capable to keep pace with these demands \cite{kim10}. This insight lead to an increased attention of the electrical engineering community as well as from applied mathematicians within the last decades. Researchers focused on the optimal design of semiconductor devices based on tailored mathematical optimization techniques \cite{diebold96, khalil95, BSFWIK93semicond2, PinHin, BuPi03, cheng2011recovering, li2013geometric,Patent12}. In fact, several design questions were considered, such as increasing the current during on-state, decreasing the leakage current in the off-state or shrinking the size of the device \cite{BuPi03,BuPiWo08,drago2008optimal,HiPi00b}. 

Meanwhile, there is a good understanding of the mathematical questions concerning the underlying optimization methods for different semiconductor models, like the drift-diffusion or the energy transport model (for the specific models see also \cite{Moc83,MaRiSch90} and the references therein). Further, fast and reliable numerical algorithms were designed on basis of the special structure of the device models, which use adjoint information to provide the necessary derivative information \cite{BuPi03,HiPi07,BuPiWo08,MaPi12,drago2013semiconductor,BuPi09}. But not only the classical model hierarchy was used, also macroscopic quantum models, e.g., the quantum drift-diffusion model \cite{UnVo07,burger2014optimal} and the quantum Euler-Poisson model \cite{pinnau2014semi} were investigated. 

Special interest is in the design or identification  of the doping profile of the charged background ions, which is most important for the electrical behaviour of the semiconductor device \cite{FWIK92semicond1,BSFWIK93semicond2,HiPi00}. The specific structure of the semiconductor models stemming from the nonlinear coupling with the Poisson equation for the electrostatic potential allows it also to use the total space charge as a design variable, which lead to the construction of a fast and optimal design algorithm in the spirit of the well-known Gummel-iteration \cite{Gum64,BuPi03,BuPi09,Patent12}.

For the semiconductor model hierarchy, it is well known that all these models are linked by asymptotic limits, which were thoroughly investigated during the last decades (for an overview see \cite{Moc83,Mar86,MaRiSch90,Jun11} and the references therein). Especially, they were used for the derivation of approximate models and algebraic formulas describing the device behaviour, like current-voltage or capacity-voltage characteristics \cite{Sze}.

One particular limit is the quasi-neutral limit in the classical drift-diffusion model for small Debye length, which is exploited for the construction of analytical current-voltage characteristics \cite{MaRiSch90,Sze}. For the forward problem this limit is analytically well understood (see, e.g.~,\cite{MaRiSch90,Unterreiter,Unt94,neutral1,neutral2,neutral3,neutral4} and the references therein). For vanishing Debye length one obtains the so-called zero space charge approximation, which has been also used in the reconstruction of semiconductor doping profiles from a Laser-Beam-Induced Current Image (LBIC) in \cite{FangIto}. Also the fast optimization approach in \cite{BuPi03,BuPi09} suggests that this limit is of particular interest for optimal semiconductor design.  

Here we investigate, if it is reasonable to approximate the solution of the drift-diffusion (DD) model for small Debye length with the zero space charge solution on the whole domain during optimal design calculations for semiconductor device. This will significantly speed up the design calculations, since instead of the multiple solution of a nonlinear partial differential equation it just requires several solutions of an algebraic equation \cite{MaRiSch90}. The answer to this question is positive and underlined for the first time by analytical and numerical results for the full optimization problem. In particular, we consider the quasi-neutral limit for a  PDE constrained optimization problem governed by the DD model in thermal equilibrium \cite{Unterreiter,Unt94}. Using the concept of $\Gamma$-convergence we can perform the asymptotic limit and can even show the convergence of minima and minimizers \cite{Braides,DalMaso}. This gives an analytical foundation for the assumptions in \cite{FWIK92semicond1} and justifies the future usage of the space-mapping approach in optimal semiconductor design (compare also \cite{drago2008optimal,MaPi12}).

The mathematical challenges are on the one hand the rigorous analysis with reduced regularity assumptions for the forward problem, and on the other hand the non-convexity of the underlying optimization problem, which allows for the non-uniqueness of the minimizer. To tackle the first problem, we use the dual formulation of the variational approach in \cite{Unterreiter,Unt94}, which allows to formulate the optimal design problem as a bi-level optimization problem in the primal variable given by the electrostatic potential. This yields then an optimization problem constrained by a nonlinear, nonlocal Poisson equation (NNPE). The second challenge suggests that we cannot expect any rates for the asymptotic limit. Hence, we rely here on the weak concept of $\Gamma$-convergence.

The paper is organized as follows. In the remainder of this section we describe the model equations and the corresponding constrained optimization problems. In the next section, we provide a thorough analysis for the state equation given by the NNPE. Using a variational approach we show existence and uniqueness of the state, as well as a priori estimates necessary for the asymptotic limit. In Section 3 we investigate the design problem analytically. The quasi-neutral limit for the full optimal design problem is performed in Section 4, where we show the $\Gamma$-convergence of the minima and minimizers. In the last section we present reliable numerical algorithms for the solution of the NNPE as well as for the adjoint problem, which are used for the construction of a descent algorithm for the optimization problem. The presented numerical results underline the analysis in the previous sections. Finally, we give concluding remarks.

\subsection{The model equations and design problem}

The scaled drift-diffusion equations in thermal equilibrium  \cite{MarkovichRinghoferSchmeiser} are given in dimension $d=1,2 $ or 3 either by the coupled system
\begin{gather*}
 n\nabla V + \nabla n =0,\qquad
 -p\nabla V + \nabla p = 0\\
 -\lambda^2\Delta V = n - p - C
\end{gather*}
or equivalently, by the nonlinear, nonlocal Poisson equation (NNPE) on $\Omega\subset\mathbb{R}^d$ 
\begin{equation}\tag{$\text{\bf P}_\lambda$}\label{NPE}
 -\lambda^2 \Delta V = n(V) - p(V) - C,
\end{equation} 
where the unknown $V$ is the electrostatic potential, $C$ is a given doping profile (later the design variable), $\lambda$ is the scaled Debye length and the charge carriers are described by the densities of electrons and holes, respectively, defined by
\begin{align}\label{eq:dens}
n= n(V) = N\frac{e^{-V}}{\int e^{-V}\,dx}, \qquad p=p(V) = P\frac{e^V}{\int e^V\,dx},
\end{align}
for some given $N,P\ge 0$. The doping profile may be split into its positive and negative part $C^+$ and $C^-$ describing the distributions of positively and negatively charged background ions. Then, it holds $C = C^+ + C^-$ and we can define the positive and negative total charges
\begin{align}\label{eq:NP}
 N = \delta^2 + \int_\Omega C^+ \,dx,\qquad P= \delta^2 -\int_\Omega C^- \,dx,
\end{align}
respectively, where $\delta^2>0$ is the so-called scaled intrinsic density of the semiconductor \cite{MaRiSch90,Unterreiter}. 

Since the device is in thermal equilibrium, we supplement \eqref{NPE} with homogeneous Neumann boundary conditions $\nu \cdot \nabla V = 0$ on $\partial\Omega$, where $\nu$ is the outward unit normal along $\partial\Omega$. Notice that the boundary condition is consistent with \eqref{NPE}. Indeed, we have
\[
 -\lambda^2\int_{\partial\Omega}\nu\cdot\nabla V\,ds = -\lambda^2\int_\Omega \Delta V\,dx = N - P - \int C\,dx = 0,
\]
which is the global space charge neutrality. For reasons of uniqueness, we further impose the integral constraint $\int_\Omega V\,dx=0$ (compare \cite{Unterreiter}).

In the quasi-neutral limit $\lambda \to 0 $, \eqref{NPE} reduces to the algebraic equation
\begin{align}\tag{$\text{\bf P}_0$}\label{zero}
 B(V) := -n(V) + p(V) + C = 0.
\end{align}

\begin{rem}
Notice that $B(k)=(P-N)/|\Omega| + C$ for any constant $k\in\mathbb{R}$ and $\int_\Omega B(V)\,dx=0$. Therefore, imposing the integral constraint $\int_\Omega V\,dx=0$ is justified.
\end{rem}

Unless otherwise stated, we make the following assumptions throughout the manuscript.
\begin{enumerate} 
 \item[(A1)] $\Omega \subset \mathbb{R}^d$, $d=1, 2$ or $3$ is a bounded Lipschitz domain.
 \item[(A2)] The doping profile satisfies $C \in H^1(\Omega)$.
\end{enumerate}
We use the abbreviation $\norm{\cdot}_p = \norm{\cdot}_{L^p(\Omega)}$.

\begin{rem}
Existence and uniqueness results for the nonlinear Poisson problem without the nonlocal terms and with different boundary conditions can be found in \cite{MaRiSch90} and the references therein. Further, a dual variational approach was used in \cite{Unterreiter,Unt94} to incorporate the Neumann boundary conditions as well as the constraints \eqref{eq:dens}. In both approaches, the analysis requires $C \in L^\infty(\Omega)$, which would be too strict for our requirements. Hence, we use instead assumption (A2) and the dual, nonlocal formulation, which will yield better estimates necessary for our asymptotic analysis. 
\end{rem}

The optimal design approach based on the fast Gummel iteration considered in \cite{BuPi03,BuPi09} suggests that the overall device behaviour is determined by the total charge $n(V)-p(V)-C$.
Hence, we consider in the following a design problem described by a cost functional of tracking-type given by
\[
 J(V,C) = \frac{1}{2} \norm{n(V) - n_d}_2^2 + \frac{1}{2} \norm{p(V) - p_d}_2^2 + \frac{\sigma}{2} \norm{\nabla (C - C_{\text{ref}})}_2^2,
\]
where $n_d$ and $p_d$ are desired electron and hole distributions, respectively. Here, $C_{\text{ref}}\in H^1(\Omega)$ is a given reference doping profile (see also \cite{HiPi00}) and $\sigma>0$ is a parameter, which allows to adjust the deviation from this reference profile \cite{Pinnau}. This allows to adjust the negative and positive charges separately by changing the doping profile $C$. Note, that the cost term involves the $H^1$-seminorm, which is essential for asymptotic limit later on.

Now we are in the position to formulate the optimization problems under consideration:
\begin{align}\tag{$\text{\bf OP}_\lambda$}\label{OP}
\text{min\, } J(V,C)\quad\text{subject to\, \eqref{NPE}}.
\end{align}
The asymptotic results for the hierarchy of semiconductor models suggest that there should also be an asymptotic link between the optimization problems \eqref{OP} for $\lambda>0$ and $\lambda =0$. In particular, we are interested in the convergence of minimizing pairs $\{(V_\lambda,C_\lambda)\}$ towards $(V_0,C_0)$, as well as in the convergence of $J(V_\lambda,C_\lambda)$ towards $J(V_0,C_0)$. For small $\lambda$ one can then use the reduced, algebraic model for the optimization, which will significantly speed-up the optimization process.

\section{The Nonlocal Nonlinear Poisson Problem}\label{sec:state}

In this section we provide a priori estimates for the solutions of \eqref{NPE}, $\lambda\ge 0$, and thereafter show existence, as well as uniqueness of solutions under assumptions (A1) and (A2).

Consider the NNPE \eqref{NPE} given by
 \[
  -\lambda^2\Delta V + B(V) = 0 \quad\text{in}\;\;\Omega,\qquad \nu\cdot \nabla V=0\quad\text{on}\;\;\partial\Omega,
 \]
with integral constraint $\int_\Omega V\,dx =0$, where as before
 \[
  B(v):= -N\frac{e^{-v}}{\int e^{-v}\,dx} + P\frac{e^v}{\int e^v \,dx} + C.
 \]
 For the analysis we reformulate \eqref{NPE} using a variational approach: Consider the functional
 \[
  b(v) := N\ln\left( \frac{1}{|\Omega|}\int_\Omega e^{-v}\,dx \right) + P\ln\left( \frac{1}{|\Omega|}\int_\Omega e^{v}\,dx \right) + \int_\Omega C\, v\,dx.
 \]
One readily sees that the first variation of the functional $b$ formally gives the operator $B$. Indeed, taking the variation of $b$ at some $v\in \C_0^\infty(\Omega)$, we obtain
 \[
  \frac{\delta}{\delta v}b(v)[h] = \int_\Omega B(v) h\,dx\quad\text{for all $h\in\C_0^\infty(\Omega)$}.
 \]
Again, since $b(k)=0$ for any constant $k\in\rr$, it suffices to consider functions with $\int v\,dx=0$. 

Henceforth, we set
$
 \P = \{ v\in L^1(\Omega)\,|\, \int_\Omega v\,dx=0 \},
$
and consider an extension of $b$ given by 
 \[
 \bar b_\lambda\colon L^1(\Omega)\to \rr\cup\{+\infty\}:\quad  \bar b_\lambda(v) = \begin{cases}
   \lambda^2\|\nabla v\|_{2}^2 + b(v) & v\in \Sigma \\
   +\infty & \text{else}
  \end{cases},
 \]
for any $\lambda\ge 0$, with $\Sigma := \{ v\in \P\;|\; b(v)<+\infty,\; \nabla v\in L^2(\Omega)\}$. 

\begin{rem}
Note, that the first variation of $\bar b_\lambda$ for $v\in\Sigma$ is in fact the weak formulation of nonlinear Poisson equation \eqref{NPE} (cf.~\eqref{eq:weakNPE}).
\end{rem}

The main result of this section is summarized in the following theorem.

\begin{thm}\label{thm:main_state}
 Let $C\in H^1(\Omega)$. For each $\lambda\ge 0$, there exists a unique minimizer $V_\lambda$ of the problem
 \begin{align}\tag{$\text{\bf MP}_\lambda$}\label{min_lambda}
  \min\nolimits_{v\in L^1(\Omega)} \bar b_\lambda(v),
 \end{align}
and consequently a unique solution of the nonlinear Poisson equation \eqref{NPE} with
 \[
 V_\lambda\in \Sigma\cap L^\infty(\Omega),\; \lambda>0, \quad \text{ and } \quad V_0\in\Sigma.
 \]
 Furthermore, the sequence $(V_\lambda)\subset\Sigma$ satisfies the following convergences
 \begin{align*}
  V_\lambda\rightharpoonup V_0\quad\text{in\, $H^1(\Omega)$},&\qquad V_\lambda\to V_0\quad\text{in\, $L^4(\Omega)$},\\
  n(V_\lambda) \to n(V_0),&\qquad p(V_\lambda)\to p(V_0)\quad\text{in\, $L^2(\Omega)$},
 \end{align*}
 where $n$ and $p$ are as given in \eqref{eq:dens}.
\end{thm}

\begin{rem}
The main contribution of the result above is the well-posedness of \eqref{NPE} also for unbounded doping profiles $C\in H^1(\Omega)$, which allows for more general control functions in the optimal control problem discussed in the following sections. This clearly extends the results in \cite{Unterreiter}. Furthermore, we are able to deduce $L^\infty$-estimates for the case $\lambda>0$ and $C\in L^p(\Omega)$, $p>2$, which was up to the authors knowledge also not known before in this setup.
\end{rem}

We outline the idea of the proof: To prove the first part of the theorem, we invoke a standard technique of variational calculus \cite{struwe} on a family of auxiliary problems. Namely, we consider the minimization of the auxiliary functional given by
 \begin{align}\tag{$\text{\bf MPa}_\lambda$}\label{OPaux}
  \bar b_\lambda^\sigma\colon L^1(\Omega)\to \rr\cup\{+\infty\}:\quad \bar b_\lambda^\sigma(v) = \begin{cases}
   \lambda^2\|\nabla v\|_{2}^2 + b(v) & v\in \Sigma_\sigma \\
   +\infty & \text{else}
  \end{cases},
 \end{align}
 for $\lambda,\sigma\ge 0$, with $\Sigma_\sigma := \big\{ v\in \Sigma\;|\; \|\nabla v\|_{2} \le \sigma \big\}\subset\Sigma$.

 To this end, we will show in Subsection~\ref{subsec:properties} that $\bar b_\lambda^\sigma$ is strictly convex, coercive and weakly lower semicontinous on $L^1(\Omega)$. We then derive necessary a priori estimates for weak solutions of \eqref{NPE}, $\lambda\ge 0$ in Subsection~\ref{subsec:apriori}, which will allow us to obtain unique minimizers for $\bar b_\lambda$.

\subsection{Properties of the Functionals}\label{subsec:properties}

\begin{lem}\label{lem:main_coercive}
 The functionals $\bar b_\lambda^\sigma$  are coercive in $L^1(\Omega)$ for $\lambda\ge 0$, $\sigma> 0$.
\end{lem}
\begin{proof}
Let $v\in \Sigma_\sigma$ and $v\ne 0$. Throughout the proof, we denote $\Omega^\pm := \text{supp}(v^\pm)$ to be the support of $v^+=\max\{0,v\}$ and $v^-=\min\{v,0\}$ respectively. Since, $v\ne 0$, we have $\Omega^\pm\ne\emptyset$, simply due to the integral constraint $\int_\Omega v\,dx=0$. Furthermore, we set $u_{A} = \frac{1}{|A|}\int_{A} u\,dx$ for any measurable set $A\subset\Omega$.

 We begin with an elementary result due to Jensen's inequality. Since
 \[
  \frac{1}{|\Omega|}\int_\Omega e^{\pm v}dx \ge \frac{1}{|\Omega|}\int_{\Omega^\pm} e^{\pm v^\pm}dx = \frac{|\Omega^\pm|}{|\Omega|}\frac{1}{|\Omega^\pm|}\int_{\Omega^\pm} e^{\pm v^\pm}dx,
 \]
 we obtain 
 \begin{align*}
  \ln\left( \frac{1}{|\Omega|}\int_\Omega e^{\pm v}dx \right) &\ge \ln\left( \frac{1}{|\Omega^\pm|}\int_{\Omega^\pm} e^{\pm v^\pm}dx \right) + \ln(|\Omega^\pm|/|\Omega|)\\
  &\ge \frac{1}{|\Omega^\pm|}\int_{\Omega^\pm} |v^\pm|\,dx + \ln(|\Omega^\pm|/|\Omega|) =: |v^\pm|_{\Omega^\pm} + \eta^\pm
 \end{align*}
 simply from the monotonicity of the logarithm and Jensen's inequality for concave functions. Consequently, we have
 \[
  N\ln\left( \frac{1}{|\Omega|}\int_\Omega e^{-v}\,dx \right) \ge N \left( |v^-|_{\Omega^-} + \eta^- \right) = \int_\Omega C^+ \big(|v^-|_{\Omega^-} \big)dx + \delta^2 |v^-|_{\Omega^-} + N\eta^-.
 \]
 Similarly, we obtain for the other term
 \[
  P\ln\left( \frac{1}{|\Omega|}\int_\Omega e^{v}\,dx \right) \ge \int_\Omega |C^-| \big( |v^+|_{\Omega^+}\big)dx + \delta^2 |v^+|_{\Omega^+} +P\eta^+.
 \]
 On the other hand, we have for the linear term
 \[
  \int_\Omega Cv\,dx \ge - \int_{\Omega^+} |C^-|v^+dx -\int_{\Omega^-} C^+|v^-|\,dx.
 \]
 Putting these inequalities together yields
 \begin{align*}
   \bar b_\lambda^\sigma(v) &\ge \lambda^2\|\nabla v\|_{2}^2 + \frac{\delta^2}{|\Omega|}\|v\|_{1} + (N\eta^- + P\eta^+)\\
   &\hspace{5em} - \int_{\Omega^-} C^+ \Big[|v^-| - |v^-|_{\Omega^-} \Big]dx - \int_{\Omega^+} |C^-| \Big[|v^+| - |v^+|_{\Omega^+}\Big]dx.
 \end{align*}
 To estimate the last two terms, we use the Poincar\'e inequality, i.e.,
 \[
  \left\| u - u_{A} \right\|_{L^2(A)} \le c_p\|\nabla u\|_{L^2(A)},
 \]
 with $u\in\{|v^-|,|v^+|\}$ and $A\in\{ \Omega^-,\Omega^+ \}$ respectively. 
  
 Since $v\in \Sigma_\sigma$, $\|\nabla v\|_{2}\le \sigma$, we finally obtain
 \begin{align}\label{lem:main_coercive:eq:co}
  \bar b_\lambda^\sigma(v) \ge \lambda^2\|\nabla v\|_{2}^2 + \frac{\delta^2}{|\Omega|}\|v\|_{1}  - \sigma(c_p^+ + c_p^-)\|C\|_{2} + (N\eta^- + P \eta^+),
 \end{align}
 which yields the coercivity in $L^1(\Omega)$ and thereby concluding the proof.
\end{proof}

\begin{rem}
 As a matter of fact, in the case $\lambda>0$, coercivity of the functional $\bar b_\lambda$ may be obtained directly. Indeed, applying Young's inequality on \eqref{lem:main_coercive:eq:co} gives
 \[
  \bar b_\lambda^\sigma(v) \ge (\lambda^2-\kappa)\|\nabla v\|_{2}^2 + \frac{\delta^2}{|\Omega|}\|v\|_{1}  - \frac{c_p^+ + c_p^-}{\kappa}\|C\|_{2}^2 + (N\eta^- +P\eta^+),
 \]
 for any $\kappa>0$. Choosing $\kappa<\lambda^2$ provides the coercivity in $L^1(\Omega)$, and in fact, also in $H^1(\Omega)$. Therefore, the restriction $\|\nabla V\|_{2}\le \sigma$ is superfluous in this case.
\end{rem}

A direct consequence of the proof of Lemma~\ref{lem:main_coercive} is the following result. Its proof is a slight modification of the arguments used in Lemma~\ref{lem:main_coercive}, which we therefore omit.

\begin{cor}\label{cor:regular}
 If $v\in\Sigma$ satisfies $b(v)\le M$ for some $M>0$, then there exists a constant $K>0$, depending only on $M$ and $|\Omega|$ such that
 \[
  \int_\Omega e^{\pm v}dx \le |\Omega|\exp\left(\frac{1}{\delta^2}\Big(K + (c_p^+ + c_p^-)\|C\|_{2}\|\nabla v\|_{2}\Big)\right).
 \]
 In particular, $n, p$ and $e^{|v|}$ are elements of $L^1(\Omega)$.
\end{cor}

\begin{lem}\label{lem:main_wlsc}
 The functionals $\bar b_\lambda^\sigma$  are weakly lower semicontinuous in $L^1(\Omega)$ for $\lambda\ge 0$, $\sigma> 0$. More precisely, we have that
 \[
  v_n\rightharpoonup v\quad\text{in $L^1(\Omega)$}\quad \Longrightarrow \quad \bar b_\lambda^\sigma(v) \le \liminf\nolimits_{n\to \infty}\bar b_\lambda^\sigma(v_n).
 \]
\end{lem}
\begin{proof}
 We consider the case where $\liminf_{n\to\infty}\bar b_\lambda^\sigma(v_n) <+\infty$, otherwise there is nothing to show. In this case we may extract a bounded subsequence (not relabeled) with $(v_n)\subset\Sigma_\sigma$. In particular, $\|\nabla v_n\|_{2}\le \sigma$ for all $n\in\nn$. Consequently, $(v_n)\subset H^1(\Omega)$ is bounded by the Poincar\'e inequality, which tells us that $v_n\to v$ in $L^2(\Omega)$ and $\nabla v_n\rightharpoonup \nabla v$ in $L^2(\Omega)$, for some subsequence (not relabeled). Hence, for the terms
 \[
  \int_\Omega Cv_n\,dx \quad\text{and}\quad \int_\Omega |\nabla v_n|^2dx,
 \]
 weak lower semicontinuity follows from linearity and the properties of norms, respectively.  Therefore, we are left to show the weak lower semicontinuity of the two middle terms of $\bar b_\lambda^\sigma$. 
 
 This result ultimately follows from the weak lower semicontinuity of the functionals $\int_\Omega e^{\pm v}dx$ in $L^1(\Omega)$ and the continuity and monotonicity of the logarithm function. In the following, we define the terms $s_n^\pm:= \int_\Omega e^{\pm v_n} dx$ and $g_k:=\inf_{n\ge k}\{s_n\}$. Since
 \[
  \int_\Omega e^{\pm v}dx \le \liminf\nolimits_{n\to \infty}\int_\Omega e^{\pm v_n}dx,
 \]
 we are left to show that $\ln(\liminf_{n\to\infty} s_n) \le \liminf_{n\to\infty} \ln(s_n)$. By definition,
 \[
  g_k\le s_n\quad\forall n\ge k\quad\Longrightarrow\quad \ln(g_k) \le \ln(s_n)\quad\forall n\ge k,
 \]
 since $\ln$ is monotonically increasing. Consequently, $\ln(g_k) \le \inf\nolimits_{n\ge k} \{\ln(s_n)\}$. Hence
 \begin{align*}
  \ln(\liminf_{n\to\infty} s_n) = \ln(\lim\nolimits_{k\to\infty} g_k) &= \lim_{k\to\infty} \ln(g_k) \\
  &\le \lim\nolimits_{k\to\infty}\inf\nolimits_{n\ge k} \{\ln(s_n)\} = \liminf_{n\to\infty} \ln(s_n),
 \end{align*}
 due to the continuity of the logarithm.
\end{proof}

\begin{lem}\label{lem:main_convex}
 The functional $\bar b_\lambda$ is strictly convex for all $\lambda\ge 0$.
\end{lem}
\begin{proof}
The first part of $\bar b_\lambda$ for $\lambda > 0$ is strictly convex, simply due to the strict convexity of the norm. Therefore, it suffices to show the strict convexity of $\bar b_0$. 

For $u,v \in \Sigma$, $\kappa \in (0,1)$, we obtain from the H\"older inequality
\[
 \int_\Omega e^{\pm(\kappa u + (1-\kappa)v)} dx \leq \norm{e^{\pm\kappa u}}_{\sigma} \norm{e^{\pm(1-\kappa)v}}_{\frac{\sigma}{\sigma - 1}}
 = \norm{e^{\pm u}}_{\kappa \sigma}^\kappa \norm{e^{\pm v}}_{(1-\kappa)\frac{\sigma}{\sigma - 1}}^{1-\kappa}.
\]
Choosing $\sigma = 1/\kappa$ then yields
\[
 \int_\Omega  e^{\pm(\kappa u + (1-\kappa)v)} \,dx \leq \norm{e^{\pm u}}_{1}^\kappa \norm{e^{\pm v}}_{1}^{1-\kappa}.
\]
Since the logarithm is monotonically increasing we get
\begin{align*}
 \ln \left(\int_\Omega e^{\pm(\kappa u + (1-\kappa)v)} dx\right ) &\leq \ln \left(\norm{e^{\pm u}}_{1}^\kappa \norm{e^{\pm v}}_{1}^{1-\kappa} \right) \\
 &= \kappa \ln \left( \int_\Omega e^{\pm u} dx\right) + (1-\kappa) \ln \left(\int_\Omega e^{\pm v} dx \right).
\end{align*}
Since $\int C v\,dx$ is linear, we obtain altogether the convexity of $\bar b_0$. 

To ensure the strict convexity, we show that equality holds only for $u = v$. Since H\"older's inequality is based on Young's inequality, equality holds for 
\[
 \frac{e^{\kappa u\sigma }}{\norm{e^{\kappa u}}_{\sigma}^\sigma} = \frac{e^{(1-\kappa)v\frac{\sigma}{\sigma -1} }}{\norm{e^{(1-\kappa)v}}_{\frac{\sigma}{\sigma -1}}^{\frac{\sigma}{\sigma -1}}}.
\]
Setting again $\sigma = 1/\kappa$, we have
\[
 \frac{e^{\kappa u\sigma }}{\norm{e^{\kappa u}}_{\sigma}^\sigma} = \frac{e^u}{\norm{e^{\kappa u}}_{\sigma}^\sigma} \quad\text{and}\quad \frac{e^{(1-\kappa)v\frac{\sigma}{\sigma -1} }}{\norm{e^{(1-\kappa)v}}_{\frac{\sigma}{\sigma -1}}^{\frac{\sigma}{\sigma -1}}} = \frac{e^v}{\norm{e^{(1-\kappa)v}}_{\frac{\sigma}{\sigma -1}}^{\frac{\sigma}{\sigma -1}}}.
\]
Thus, in the equality case
\[
 e^u = e^v \frac{\norm{e^{\kappa u}}_{\sigma}^\sigma}{\norm{e^{(1-\kappa)v}}_{\frac{\sigma}{\sigma -1}}^{\frac{\sigma}{\sigma -1}}}\quad\Longleftrightarrow\quad
 u = v + w \quad \text{with } \quad w = \ln \left( \frac{\norm{e^{\kappa u}}_{\sigma}^\sigma}{\norm{e^{(1-\kappa)v}}_{\frac{\sigma}{\sigma -1}}^{\frac{\sigma}{\sigma -1}}}\right).
\]
However, from the assumptions $\int_\Omega u \,dx = 0$, $\int_\Omega v \,dx = 0$, we obtain $w = 0$ and thus $u=v$, which implies strict convexity of the functional.
\end{proof}

\subsection{A priori estimates}\label{subsec:apriori} 

We begin by showing a priori estimates for weak solutions of the algebraic equation \eqref{zero}, and proceed with $L^\infty$-estimates of solutions to \eqref{NPE}, $\lambda>0$.

\begin{lem}\label{lem:apriori_zero}
 Let $C\in H^1(\Omega)$ and $V\in \Sigma$ be a solution of \eqref{zero}, then $n$ and $p$, defined in \eqref{eq:dens}, are elements of $L^2(\Omega)$. Define for $V$ the values
 \[
  \alpha:= \frac{N}{\int_\Omega e^{-V}\,dx},\quad \beta:=\frac{P}{\int_\Omega e^V\,dx},\quad 0<\gamma^2:= \sqrt{\alpha\beta} <\infty,
 \]
 with $N,P$ given in \eqref{eq:NP}. Then, $V\in\Sigma$ satisfies further $\gamma^2\|\nabla V\|_2 \le \frac{1}{2}\|\nabla C\|_2$.
\end{lem}
\begin{proof}
The proof mimics ideas stated in \cite{Unterreiter,Unt94}.
 Set $V = \ln g$ for some nonnegative function $g$. Then the algebraic equation \eqref{zero} reads
 \[
  \frac{\gamma^4}{\beta}\frac{1}{g} - \beta g = C,
 \]
 or equivalently
 \[
  g^2(x) + g(x)C(x)/\beta - \gamma^4/\beta^2 = 0.
 \]
 Solving for $g$ gives
 \[
  g(x) = \frac{1}{2\beta}\left[ - C(x) + \sqrt{4\gamma^4 + C^2(x)}\, \right],
 \]
 since $g$ is required to be nonnegative. Therefore,
 \begin{align}\label{VinL0}
  V(x) = \ln\left( \frac{1}{2\beta}\left[ - C(x) + \sqrt{4\gamma^4 + C^2(x)}\, \right] \right).
 \end{align}
 A simple consequence of the algebraic expression is that regularity of solutions may be determined easily. Indeed, taking the exponential of $V$ and rearranging the terms give
 \begin{align}\label{eq:representation}
  p=p(V) = \beta e^{V} = \frac{1}{2}\left[ - C(x) + \sqrt{4\gamma^4 + C^2(x)}\, \right].
 \end{align}
 Therefore, we square the equation for $p$ and integrate over $\Omega$ to obtain
 \begin{align}\label{eq:p_estimate}
  \int_\Omega p^2 dx = \frac{1}{4}\int_\Omega \left|C - \sqrt{4\gamma^4 + C^2}\right|^2 dx \le \int_\Omega \big(C^2 + 2\gamma^4\big)\,dx,
 \end{align}
 which provides an $L^2$-estimate for $p\in L^2(\Omega)$. Using the algebraic equation \eqref{zero} again, we end up with a similar estimate for $n$. More specifically, we have
 \begin{align}\label{eq:algebraic_eq}
  n = C + p\in L^2(\Omega),
 \end{align}
 with a similar $L^2$-bound on $n$ as in \eqref{eq:p_estimate}. Therefore, both $n,p\in L^2(\Omega)$. 
 
 Similarly, we compute the derivative of the algebraic equation to obtain
  \[
   (C + 2\beta g)\nabla g + g\nabla C = 0,
  \]
  and consequently,
  \[
   \int_\Omega |\nabla V|^2 dx = \int_\Omega \left|\frac{\nabla g}{g}\right|^2 dx = \int_\Omega \left|\frac{1}{\sqrt{4\gamma^4 + C^2}}\right|^2|\nabla C|^2 dx.
  \]
 Since $0\le 1/|4\gamma^4 + C^2| \le 1/4\gamma^4$ a.e.~in $\Omega$, one gets the estimate for $\gamma^2\|\nabla V\|_2$. 
\end{proof}

\begin{rem}\label{rem:p_regularity}
 In fact, the regularity of $n$ and $p$ may be significantly improved. Indeed, since $H^1(\Omega)\hookrightarrow L^6(\Omega)$ for $d\le 3$, we can use \eqref{eq:representation} to obtain $p\in L^6(\Omega)$, and hence $n\in L^6(\Omega)$ from the algebraic equation \eqref{eq:algebraic_eq}. Furthermore, we have that $p\in W^{1,1}(\Omega)$, and consequently $n\in W^{1,1}(\Omega)$. Indeed, we use the representation \eqref{eq:representation} to obtain
 \[
  \int_\Omega |\nabla p|\,dx \le \frac{1}{2}\int_\Omega |\nabla C|\,dx + \int_\Omega \frac{|C|}{\sqrt{4\gamma^4 + C^2}}|\nabla C|\,dx \le \left(\frac{|\Omega|}{2} + \frac{\|C\|_2}{4\gamma^4}\right)\|\nabla C\|_2.
 \]
 Another simple observation that results from the algebraic equation \eqref{eq:algebraic_eq} is the explicit form
 \[
  n+p = C + 2p = \sqrt{4\gamma^4 + C^2} \ge 2\gamma^2\qquad\text{a.e.~in $\Omega$}.
 \]
\end{rem}

As pointed out earlier, for $\lambda>0$, we obtain uniform estimates when $C\in L^r(\Omega)$, $r>2$. The proof essentially relies on the fact that $V_\lambda\in\Sigma$ and that $B$ is monotone, since $\bar b_\lambda$ is convex on its domain of definition. To obtain $L^\infty$-estimates, we make use of the Stampacchia method \cite{kinderlehrer}.

\begin{lem}\label{lem:main_regular}
 Let $C\in L^2(\Omega)$ and $V_\lambda\in \Sigma$, $\lambda>0$, be a solution of \eqref{NPE}, then
 \[
  \|V_\lambda\|_{H^1(\Omega)} \le c_\lambda\,\|C\|_{2},
 \]
 for a constant $c_\lambda>0$, depending on $\lambda$. Furthermore, $V_\lambda\in L^\infty(\Omega)$ for $C\in L^r(\Omega)$, $r>2$.
\end{lem}
\begin{proof}
 Testing \eqref{NPE} with $V_\lambda\in\Sigma$ and integration by parts yield
 \[
  \lambda^2 \int_\Omega|\nabla V_\lambda|^2 dx + \int_\Omega B(V_\lambda)V_\lambda\,dx = 0.
 \]
 Due to the convexity of $b$, we know that $B$ is monotone, i.e.
 \[
  \langle B(u)-B(v), u - v\rangle \ge 0\quad\text{for all $u,v\in \Sigma$}.
 \]
 Therefore, inserting $B(0)$ in between yields
 \[
  \langle B(V_\lambda),V_\lambda\rangle = \langle B(V_\lambda)-B(0),V_\lambda\rangle + \langle B(0),V_\lambda\rangle \ge \langle B(0),V_\lambda\rangle = \int_\Omega CV_\lambda\,dx,
 \]
 where we used the fact that $V_\lambda\in\P$. A simple application of Young's inequality yields
 \[
  \lambda^2 \int_\Omega|\nabla V_\lambda|^2 dx \le -\int_\Omega CV_\lambda\,dx \le \frac{1}{2\kappa}\|C\|_2^2 + \frac{\kappa}{2}\|V_\lambda\|_2^2 \le \frac{1}{2\kappa}\|C\|_2^2 + \frac{\kappa}{2} c_p^2\|\nabla V_\lambda\|_2^2,
 \]
 where we used the Poincar\'e inequality with constant $c_p>0$ in the last inequality. Hence, choosing $\kappa=\lambda^2/c_p^2$ provides the required estimate.
 
 Now set $\varphi_{\lambda,k}^+=\max\{0,V_\lambda-k\}\in H^1(\Omega)$, $A_k^+=\text{supp}(\varphi_{\lambda,k}^+)$ and
 \[
  \psi_{\lambda,k}^+ = \chi_{A_k^+}\left(\varphi_{\lambda,k}^+ - \int_{A_k^+} \varphi_{\lambda,k}^+\, dx\right)\in \P \cap H^1(\Omega).
 \]
 Following the arguments above, we test \eqref{NPE} with $\psi_{\lambda,k}^+$ to obtain
 \[
  \lambda^2 \int_\Omega|\nabla \psi_{\lambda,k}^+|^2 dx + \int_{A_k^+} B(V_\lambda)\psi_{\lambda,k}^+\,dx = 0.
 \]
 By definition, $\psi_{\lambda,k}^+ = V_\lambda - m$ on $A_k^+$, where $m=k + \int_\Omega \varphi_{\lambda,k}^+\, dx\in\rr$. Therefore,
 \[
  \langle B(V_\lambda),\psi_{\lambda,k}^+\rangle = \langle B(V_\lambda)-B(m),\psi_{\lambda,k}^+\rangle + \langle B(m),\psi_{\lambda,k}^+\rangle \ge \langle B(m),\psi_{\lambda,k}^+\rangle = \int_{A_k^+} C\psi_{\lambda,k}^+\,dx,
 \]
where we used, again, that $\psi_{\lambda,k}^+\in \P$. Mimicking the arguments from above, we obtain
 \begin{align}\label{eq:linfty}
   \|\nabla\psi_{\lambda,k}^+\|_{L^2(A_k^+)}^2 \le c_\lambda^2\,\|C\|_{L^2(A_k^+)}^2,
 \end{align}
 with $c_\lambda=c_p/\lambda^2$. Henceforth, we can apply Stampacchia's strategy to obtain the $L^\infty$-estimate. On the right-hand side we estimate from above by
 \[
   \|C\|_{L^2(A_k^+)}^2 \le \|C\|_{L^{2+\epsilon}(A_k^+)}^2|A_k^+|^{\epsilon/(2+\epsilon)}.
 \]
  On the left-hand side \eqref{eq:linfty}, we use the Poincar\'e inequality on $\psi_{\lambda,k}^+$ to obtain
 \[
  \left\|\varphi_{\lambda,k}^+ - \frac{1}{|A_k^+|}\int_{A_k^+}\varphi_{\lambda,k}^+\,dx \,\right\|_{L^2(A_k^+)}^2 = \|\psi_{\lambda,k}^+\|_{L^2(A_k^+)}^2 \le c\|\nabla \psi_{\lambda,k}^+\|_{L^2(A_k^+)}^2.
 \]
 The term on the left-hand side may be explicitly written as
 \[
  \int_{A_k^+} \left|V_\lambda - k - \frac{1}{|A_k^+|}\int_{A_k^+} \varphi_{\lambda,k}^+\, dx\right|^2 dx =\int_{A_k^+} |V_\lambda - k|^2dx - \frac{1}{|A_k^+|}\left|\int_{A_k^+} \varphi_{\lambda,k}^+\, dx\right|^2.
 \]
 We now estimate the last term on the right-hand side using H\"older's inequality to obtain
 \[
  \left|\int_{A_k^+} \varphi_{\lambda,k}^+\, dx\right|^2 \le |A_k^+|^{\frac{2}{p}}\left(\int_{A_k^+} |\varphi_{\lambda,k}^+|^q dx\right)^{\frac{2}{q}} = |A_k^+|^{\frac{2}{p}}\|V_\lambda-k\|_{L^q(A_k^+)}^2 \le |A_k^+|^{\frac{2}{p}}\|V_\lambda\|_{L^q(\Omega)}^2, 
 \]
 with $1/p + 1/q=1$. Choosing $p=(2+\epsilon)/(1+\epsilon)$, i.e., $q=2+\epsilon$, we obtain
 \[
  \frac{1}{|A_k^+|}\left|\int_{A_k^+} \varphi_{\lambda,k}^+\, dx\right|^2 \le \|V_\lambda\|_{L^q(\Omega)}^2|A_k^+|^{\epsilon/(2+\epsilon)}.
 \]
 Furthermore, notice that for any $h>k\ge 0$, we have that $A_h^+\subset A_k^+$ and hence
 \[
  \int_{A_k^+} |V_\lambda - k|^2dx \ge \int_{A_h^+} |h - k|^2dx = |h-k|^2|A_h^+|.
 \]
 Putting these terms together leads to the inequality
 \[
  \zeta(h) \le \frac{c_\zeta}{|h-k|^2} \zeta(k)^\beta,
 \]
 with $\zeta(h)=|A_h^+|$, $\beta=\epsilon/(2+\epsilon)$, and some constant $c_\zeta>0$ depending on the norms $\|C\|_{2+\epsilon}$ and $\|V_\lambda\|_{2+\epsilon}$. From a lemma of Kinderlehrer and Stampacchia \cite[II.~Lemma B1]{kinderlehrer}, there exists some $K>0$ such that $\zeta(k)=|A_{k}^+|=0$ for every $k\ge K$, which clearly implies $V_\lambda \le K$ almost everywhere in $\Omega$.

 Analogously, one shows a lower bound for $V_\lambda$ by testing \eqref{NPE} with
 \[
  \psi_{\lambda,k}^- = \chi_{A_k^-}\left(\varphi_{\lambda,k}^- - \int_\Omega \varphi_{\lambda,k}^-\, dx\right)\in \P \cap H^1(\Omega),
 \]
 where $\varphi_{\lambda,k}^-=\min\{V_\lambda+k,0\}\in H^1(\Omega)$ and $A_k^-=\text{supp}(\varphi_{\lambda,k}^-)$. Altogether, $V_\lambda\in L^\infty(\Omega)$.
\end{proof}

\begin{lem}\label{lem:main_convergence}
 Let $C\in H^1(\Omega)$. The sequence of minimizers $(V_\lambda)$ of \eqref{min_lambda} contains subsequences that converge weakly in $H^1(\Omega)$ and strongly in $L^4(\Omega)$ towards the unique minimizer $V_0$ of $({\rm {\bf MP}}_0)$. Furthermore,
 \[
  n(V_\lambda) \to n(V_0),\qquad p(V_\lambda)\to p(V_0)\quad\text{in\, $L^2(\Omega)$}.
 \]
\end{lem}
\begin{proof}
We begin by considering the sequence of inequalities
\begin{align}\label{eq:sequence}
 \lambda^2 \int_\Omega |\nabla V_\lambda|^2 dx + b(V_\lambda) \le \lambda^2 \int_\Omega |\nabla V_0|^2 dx + b(V_0) \le \lambda^2 \int_\Omega |\nabla V_0|^2 dx + b(V_\lambda).
\end{align}
Therefore, from Lemma~\ref{lem:apriori_zero} we obtain
\[
 \int_\Omega |\nabla V_\lambda|^2 dx \le \int_\Omega |\nabla V_0|^2 dx \le M,
\]
for some constant $M>0$. Since $(V_\lambda)\subset\P$, we have the boundedness of $(V_\lambda)$ in $H^1(\Omega)$ by the generalized Poincaré inequality. Therefore, we can extract a weakly converging subsequence (not relabeled), which converges towards some $V_*\in H^1(\Omega)$. We can then estimate
\[
 \bar b_0(V_*) \le \liminf_{\lambda \to 0} \bar b_0(V_\lambda) \le \liminf_{\lambda \to 0} \bar b_\lambda(V_\lambda) \le \liminf_{\lambda \to 0} \bar b_\lambda(V_0) = \bar b_0(V_0).
\]
However, by the uniqueness of the minimizer of $\bar b_0$, we obtain $V_*\equiv V_0$. Hence, $V_\lambda \rightharpoonup V_0$ in $H^1(\Omega)$.

Due to the compact embedding $H^1(\Omega) \hookrightarrow L^4(\Omega)$ for $d\le 3$, we may further extract another subsequence (not relabeled) satisfying
\[
 V_\lambda \to V_0\quad\text{in}\quad L^4(\Omega),\qquad V_\lambda\to V_0\quad\text{a.e.~in\, $\Omega$}.
\]
Notice from the sequence of inequalities \eqref{eq:sequence}, that
\[
 b(V_\lambda) \le \lambda^2 \int_\Omega |\nabla V_0|^2 dx + b(V_0) \le K,
\]
for some $K>0$ independent of $\lambda$. Corollary~\ref{cor:regular} then provides the boundedness of the sequence $(e^{\pm V_\lambda})$ in $L^1(\Omega)$. Moreover, the almost everywhere convergence of $V_\lambda$ towards $V_0$ also provides the almost everywhere convergence of $e^{\pm V_\lambda}$ towards $e^{\pm V_0}$. By the Lebesgue dominated convergence theorem, we obtain the strong convergence $e^{\pm V_\lambda}\to e^{\pm V_0}$ in $L^1(\Omega)$. It is now easy to see that
\[
 n(V_\lambda) \to n(V_0),\qquad p(V_\lambda)\to p(V_0)\quad\text{in}\quad L^1(\Omega).
\]
In order to show the above convergence in $L^2(\Omega)$, we will have to work slightly more. 

We begin by computing the difference
\begin{align*}
 n(V_\lambda) - n(V_0) &= \frac{N\,e^{-V_\lambda}}{\int e^{-V_\lambda}dx}-\frac{N\,e^{-V_0}}{\int e^{-V_0}dx} \\
 &= \frac{N}{\int e^{-V_\lambda}dx}\left[\left( e^{-V_\lambda} - e^{-V_0} \right) + \left( \int_\Omega e^{-V_0}dx - \int_\Omega e^{-V_\lambda}dx \right)\frac{n(V_0)}{N}\right].
\end{align*}
Taking the square of the equality above and integrating over $\Omega$ yields
\[
 \|n(V_\lambda) - n(V_0)\|_{2}^2 \le \frac{2N^2}{\|e^{-V_\lambda}\|_{1}^2}\left[\| e^{-V_\lambda} - e^{-V_0} \|_{2}^2 + \frac{\|n(V_0)\|_{2}^2}{N^2}\left| \|e^{-V_0}\|_{1} - \|e^{-V_\lambda}\|_{1}\right|^2\right].
\]
For the first term on the right-hand side, we use convexity of $e^{-s}$, $s\in\rr$, to deduce
\[
 e^{-V_0} - e^{-V_\lambda} \le e^{-V_0}(V_\lambda-V_0).
\]
Consequently, we obtain
\[
 \|n(V_\lambda) - n(V_0)\|_{2}^2 \le \frac{2N^2}{\|e^{-V_\lambda}\|_{1}^2}\left[\|e^{-V_0}\|_{4}^2\|V_\lambda-V_0\|_{4}^2 + \frac{\|n(V_0)\|_{2}^2}{N^2}\left| \|e^{-V_0}\|_{1} - \|e^{-V_\lambda}\|_{1}\right|^2\right].
\]
From Remark~\ref{rem:p_regularity}, we see that $\|e^{-V_0}\|_{4}$ is bounded, and so we can pass to the limit to conclude
\[
 n(V_\lambda)\to n(V_0)\quad\text{in}\quad L^2(\Omega).
\]
Similar arguments may be used to derive the strong convergence for $(p(V_\lambda))$ in $L^2(\Omega)$.
\end{proof}

\subsection{Proof of Theorem~\ref{thm:main_state}}
 We begin the proof by showing the existence of minimizers for the auxiliary problem \eqref{OPaux} and use a priori estimates obtained above to conclude the result for \eqref{min_lambda}.

 Let $(v_n^{\lambda,\sigma})\subset L^1(\Omega)$ be a minimizing sequence of $\bar b_\lambda^\sigma$. In particular, there exists a subsequence (not relabeled) with $(v_n^{\lambda,\sigma})\subset \Sigma_\sigma$. Due to the coercivity of $\bar b_0^\sigma$ on $L^1(\Omega)$ (cf.~Lemma~\ref{lem:main_coercive}), we have the boundedness of $(v_n^{\lambda,\sigma})$ in $L^1(\Omega)$. Furthermore, since $\|\nabla v_n^{\lambda,\sigma}\|_{2}\le \sigma$ and $\int_\Omega v_n^{\lambda,\sigma} dx=0$ for all $n\in\nn$, the generalized Poincar\'e inequality provides the boundedness of $(v_n^{\lambda,\sigma})$ in $L^2(\Omega)$, and therefore the boundedness of $(v_n^{\lambda,\sigma})$ in $H^1(\Omega)$. Due to reflexivity of $H^1(\Omega)$, we may extract a weakly converging subsequence (not relabeled), satisfying $v_n^{\lambda,\sigma}\rightharpoonup v_\lambda^\sigma$ in $H^1(\Omega)$ for some $v_\lambda^\sigma\in H^1(\Omega)$. Consequently, $v_n^{\lambda,\sigma}\rightharpoonup v$ in $L^1(\Omega)$. We further obtain $v_\lambda^\sigma\in\Sigma_\sigma$ from the weak lower semicontinuity of $\bar b_\lambda^\sigma$ in $L^1(\Omega)$ (cf.~Lemma~\ref{lem:main_wlsc}). Since $\bar b_\lambda^\sigma$ is strictly convex (cf.~Lemma~\ref{lem:main_convex}), $v_\lambda^\sigma\in\Sigma_\sigma$ is the unique minimizer.
 
 From the a priori estimate on a solution to \eqref{zero}, $\gamma^2\|\nabla V_\lambda\|_{2}\le \frac{1}{2}\|\nabla C\|_{2}$ given in Lemma~\ref{lem:apriori_zero}, we may choose $\sigma$ sufficiently large so as to obtain $V_0=v_0^\sigma$ for any $\sigma\ge \sigma_*$ for some $\sigma^*<\infty$ sufficiently large, thereby obtaining the unique minimizer of $({\bf MP}_0)$. Similarly, the same arguments apply to the case $\lambda>0$ due to the a priori estimates obtained in Lemma~\ref{lem:main_regular}. Finally, the strong convergence $V_\lambda \to V_0$ in $L^4(\Omega)$ is given in Lemma~\ref{lem:main_convergence}, thereby concluding the proof of Theorem~\ref{thm:main_state}.\hfill$\square$

\section{Analysis of the Constrained Optimization Problem}\label{sec:opt}
The goal of this section is to prove the existence of a solution to the constrained optimization problem \eqref{OP} and to derive the first-order optimality system.

\subsection{Existence of minimizers to the constraint optimization problem}
 We define the set of admissible doping profiles as
\begin{align}\label{admissible}
 \U \subset H^1(\Omega),
\end{align}
with the additional property: 
\[
 \text{$A\subset\U$ bounded w.r.t.~$\|\nabla\cdot\|_2$ implies $A$ bounded in $H^1(\Omega)$}.
\]

\begin{rem}
 Examples of such admissible sets are $H_0^1(\Omega)$, $\P\cap H^1(\Omega)$, and
 \[
  \left\{u\in H^1(\Omega)\,|\, \|u\|_{p}\le M \right\}, \quad p\ge 1,
 \]
 for some fixed constant $M>0$. This extends significantly the range of admissible sets for the doping profile considered so far \cite{HiPi00,BuPi03}, where a typical restriction was $C \in H^1(\Omega)\cap L^\infty(\Omega)$.
\end{rem}

In order to formulate the problem, we write a weak solution of the NNPE \eqref{NPE} as solution of the operator equation
\begin{align}\label{eq:weakNPE}
e_\lambda(V_\lambda,C) = 0\quad \text{in}\quad \Y^* := \big(H^1(\Omega)\cap \P\big)^*,
\end{align}
where the operator $e_\lambda \colon \Sigma \times \U \to \Y^*$ is given by 
\[
 \langle e_\lambda(V_\lambda,C),\varphi\rangle := \langle -\lambda^2 \Delta V_\lambda - n(V_\lambda) + p(V_\lambda) + C,\varphi \rangle \quad \text{ for all }\varphi \in \Y,
\]
with the previously defined densities 
\begin{align}\label{eq:densities}
 n(V) = N\frac{e^{-V}}{\int_\Omega e^{-V}\,dx} , \qquad p(V) =P\frac{e^{V}}{\int_\Omega e^{V} \,dx}.
\end{align}

The optimal control problem that is investigated in the sequel reads:
\begin{problem}
Find $(V_*,C_*)\in \Sigma\times \U$ such that
\begin{align}\tag{$\text{\bf OP}_\lambda$}\label{optproblem}
 (V_*,C_*) = \arg\min J(V,C) \quad \text{subject to}\quad e_\lambda(V,C) = 0\quad\text{in}\quad \Y^*, 
\end{align}
 where the functional is given by 
\[
 J(V,C) = \frac{1}{2} \norm{n(V) - n_d}_{2}^2 + \frac{1}{2} \norm{p(V)- p_d}_{2}^2 + \frac{\sigma}{2} \norm{\nabla (C - C_{\text{ref}})}_{2}^2,
\]
for some given $n_d,p_d\in L^2(\Omega)$, $C_{\text{ref}}\in H^1(\Omega)$ and constant $\sigma>0$.
\end{problem}

The existence of a minimizer for \eqref{optproblem} is a consequence of the well-posedness result for \eqref{NPE} and the a priori estimates obtained in the previous section.
\begin{thm}
 There exists at least one solution $(V_*,C_*)\in \Sigma\times \U$ to \eqref{optproblem}, $\lambda\ge 0$.
\end{thm}
\begin{proof}
By definition $J$ is bounded from below and we can define
\[
 j := \inf\left\{J(V,C)\,|\, (V,C)\in \Sigma\times \U \right\}\ge 0.
\]
We choose a minimizing sequence $\{(V_k,C_k)\} \subset \Sigma\times \U$. Since $J$ is radially unbounded with respect to the second variable, and $\U$ is an admissible set given in \eqref{admissible}, we obtain uniform boundedness for $(C_k)$ in $H^1(\Omega)$. Lemmas~\ref{lem:apriori_zero} ($\lambda=0$) and \ref{lem:main_regular} ($\lambda>0$) then provide the uniform boundedness of $(V_k)$ in $H^1(\Omega)$ and $\{(n(V_k),p(V_k))\}$ in $L^2(\Omega)$. This allows to extract subsequences (not relabeled):
\begin{gather*}
 V_k \rightharpoonup V_*,\quad C_k \rightharpoonup C_*\quad\text{in}\quad H^1(\Omega),\\
 n(V_k)\rightharpoonup n_*,\quad p(V_k)\rightharpoonup p_*\quad\text{in}\quad L^2(\Omega).
\end{gather*}
Consequently, passing to the limit in the weak formulation of \eqref{NPE} gives
\[
 0=\langle e_\lambda(V_k,C_k),\varphi \rangle\; \longrightarrow\; \langle -\lambda^2\Delta V_*  - n_* + p_* + C_*,\varphi\rangle \quad \text{ for all } \varphi \in H^1(\Omega).
\]
Since solutions of \eqref{NPE} are known to be unique, we may identify the limit densities $n_*=n(V_*)$, and $p_*=p(V_*)$ accordingly. Therefore, $n(V_k)\rightharpoonup n(V_*)$ and $p(V_k)\rightharpoonup p(V_*)$ in $L^2(\Omega)$.

By the weak lower semicontinuity of the $L^2$-norm, we have that
\[
 J(V_*,C_*) \le \lim\inf\nolimits_{k\to\infty} J(V_k,C_k) = j,
\]
by which we conclude that $(V_*,C_*)$ solves \eqref{optproblem}.
\end{proof}

\begin{rem}\label{control-to-state-rem}
As we have seen in the previous section, solutions $V_\lambda$ of \eqref{NPE} may be characterized as unique minimizers corresponding to \eqref{min_lambda} for each $\lambda\ge 0$, respectively. In particular, Lemmas~\ref{lem:apriori_zero} and \ref{lem:main_regular} provide a priori bounds for $V_\lambda$, $\lambda\ge 0$, where the bounds strongly depend on the doping profile $C\in H^1(\Omega)$. This allows us to define the so-called {\em control-to-state map}, $\Phi\colon \U \to \Sigma$, which maps any $C$ to the unique weak solution $V_\lambda\in\Sigma$, satisfying $e_\lambda(V_\lambda,C) = 0$ in $H^1(\Omega)^*$. In fact, this map is continuous and bounded. This fact will be essential when deriving first order necessary optimality conditions in the next sections.
\end{rem}

\begin{rem}
Note, that we cannot ensure the uniqueness of the minimizer of the optimization problem  \eqref{optproblem} due to the non-convexity induced by the constraint. Hence, the convergence behaviour of minimizing pairs $(V_\lambda,C_\lambda)$ of the optimization is a priori not clear. This is in contrast to the asymptotic limit for the state equation.
\end{rem}

\subsection{Derivation of the first-order optimality system}
In the following, we assume that
\[
  U_{ad} = C_{\text{ref}} + H^1\cap \mathcal{P},
\]
where $C_{\text{ref}}\in H^1(\Omega)$. Owing to the definition of $\U$, we have that
\[
 \int_\Omega C\,dx = \int_\Omega C_{\text{ref}}\,dx\qquad\text{for any $C\in\U$}.
\]
For reasons of convenience, we denote $u:=C-C_{\text{ref}}\in \Y$.

Formally, the first-order optimality system may be derived using the standard $L^2$ approach \cite{Pinnau,Troeltzsch}. In this approach, one defines an extended Lagrangian $\mathcal{L}\colon \Sigma\times \Y\times\Y\to \rr$ associated to the constrained optimization problem, which reads
\[
 \mathcal{L}(V,u,\xi) := J(V,u+C_{\text{ref}}) + \langle e_\lambda(V,u+C_{\text{ref}}),\xi\rangle.
\]
The first-order optimality system corresponding to $\mathcal{L}$ is then given by
\[
 d\mathcal{L}(V,u,\xi)=0,
\]
with $d$ denoting the Fr\'echet derivative of the Lagrangian $\mathcal{L}$. As usual, the derivative with respect to $\xi$ yields the state system, while the adjoint system is derived by taking the derivative with respect to $V$, i.e.,
\[
 [d_V e_\lambda(V,u+C_{\text{ref}})]^*\xi = -d_V J(V,u+C_{\text{ref}}).
\]
Taking the derivative with respect to $u$ yields 
\[
 [d_u e_\lambda(V,u+C_{\text{ref}})]^*\xi = -d_u J(V,u+C_{\text{ref}}).
\]
Elementary computations lead to the adjoint system
\begin{align}\label{eq:adjoint}
 -\lambda^2\Delta\xi + K[\xi] = K_n[n-n_d] - K_p[p-p_d]\quad\text{in $\Omega$},\qquad \nu\cdot \nabla\xi=0\quad\text{on $\partial\Omega$},
\end{align}
where
\[
 K_n[\xi] = n \left( \xi - \frac{1}{N} \int_\Omega n\,\xi\,dy \right),\qquad K_p[\xi] = p \left( \xi - \frac{1}{P} \int_\Omega p\,\xi\, dy \right),
\]
and $K[\xi] = K_n[\xi] + K_p[\xi]$, and the optimality condition for $u$ given by
\begin{align}
 \sigma\Delta u = \xi\quad\text{in $\Omega$},\qquad \nu\cdot\nabla u = 0\quad\text{on $\partial\Omega$},
\end{align}
which is clearly uniquely solvable for $u\in \Y$, as a result of standard linear elliptic theory.

In order to rigorously justify the first-order optimality system, we will first show the Fr\'echet differentiability of the Lagrangian $\mathcal{L}$. We begin this step with the following statement.

\begin{lem}\label{lem:nemytskii}
 The mappings $n,p\colon\Sigma\to L^2(\Omega)$ as defined in \eqref{eq:densities} are Fr\'echet differentiable as Nemytskii-operators.
\end{lem}
\begin{proof}
Using the differentiability of the exponential function and $(\cdot)^{-1}$ on $(0,\infty)$, we compute the Gateaux derivative of $n$ in the direction $h\in\Sigma$, which yields
\begin{align*}
 \delta n(V)[h] &= K_n[h] = n(V) \left( h - \frac{1}{N} \int_\Omega n(V) h\,dy \right).
\end{align*}
Due to Remark~\ref{rem:p_regularity}, we have that $n=n(V)\in L^4(\Omega)$, and hence
\begin{align*}
 \|\delta n [h]\|_2^2 = \int_\Omega n^2 \left( h - \frac{1}{N} \int_\Omega n h\,dy \right)^2 dx \le \|n\|_4^2 \left\| \left(h - \frac{1}{N} \int_\Omega n h\,dy \right)\right\|_4^2 <\infty,
\end{align*}
for any $h\in H^1(\Omega)$, which says that $\delta n(V)\colon H^1(\Omega)\to L^2(\Omega)$ is a bounded linear operator.

Analogously, one can show the same estimates for
\[
 \delta p(V)[h] = K_p[h] = p \left( h - \frac{1}{P} \int_\Omega ph\, dy \right),
\]
and therefore $\delta p\colon H^1(\Omega)\to L^2(\Omega)$ is also a bounded linear operator.
\end{proof}

A direct consequence of the previous lemma is the Fr\'echet differentiability of the operator $e_\lambda$, which we summarize in the following result.

\begin{lem}\label{lem:F_diff}
The mapping $e_\lambda$ as defined in \eqref{eq:weakNPE} is Fréchet differentiable. The actions of the derivative at a point $(V,u) \in \Sigma\times \Y$ in a direction $h \in H^1(\Omega) $ are given by
\begin{align*}
 \langle d_Ve_\lambda(V,u + C_{\text{ref}})[h_V] , \test \rangle &=  -\lambda^2\int_\Omega \nabla h_V\cdot\nabla\test\,dx + \int_\Omega K[h_V]\,\test \,dx,
\end{align*}
 with $K$ defined as above, and
\begin{align*}
 \langle d_u e_\lambda(V,u+ C_{\text{ref}})[h_u], \test \rangle = \int_\Omega h_u \, \test \,dx,
\end{align*}
for any test function $\test\in\Y$.
\end{lem}

\begin{rem}\label{rem:self-adjoint}
 Note that the operator $K$ is self-adjoint with respect to the $L^2$ scalar product. Indeed, elementary computations gives
 \begin{align*}
  \int_\Omega K_n[h]\,\test\,dx &= \int_\Omega n \left( h - \frac{1}{N} \int_\Omega n h\,dy \right)\test\,dx = \int_\Omega n h\test\,dx - \frac{1}{N} \left(\int_\Omega n h\,dy \right)\left(\int_\Omega n\test\,dx\right) \\
  &= \int_\Omega n \left( \test - \frac{1}{N} \int_\Omega n \test\,dx \right) h\,dy = \int_\Omega h\,K_n[\test]\,dy.
 \end{align*}
 Similarly, we show that $K_p$ is self-adjoint, and hence $K$ is self-adjoint on $L^2(\Omega)$.
 
 Furthermore, we have that 
 \[
  \int_\Omega K_n[h]\,dx = \int_\Omega K_p[h] = 0\qquad\text{for any $h\in H^1(\Omega)$},
 \]
 which would necessitate the constraint $h\in \Y= H^1(\Omega)\cap \P$.
\end{rem}

In the following we show the existence of a unique solution of the adjoint problem.
\begin{lem}
 For any $\lambda > 0$, there exists a unique solution $\xi \in \Y$ for the adjoint system \eqref{eq:adjoint}. 
\end{lem}
\begin{proof}
To apply the Lax-Milgram theorem we consider the variational formulation of \eqref{eq:adjoint}:
\[
 \text{Find $\xi\in\Y$}:\quad  a(\xi,\test) = f(\test)\qquad\text{for all $\test\in\Y$},
\]
where the bilinear form $a\colon\Y\times\Y\to\rr$ and linear form $f\colon\Y\to\rr$ are defined by
\begin{align*}
 a(\xi,\test) &:= \lambda^2\int_\Omega \nabla \xi \cdot \nabla \test\,dx + \int_\Omega K[\xi]\, \test\, dx, \\
 f(\test) &:= \int_\Omega \big(K_n[n-n_d] - K_p[p-p_d]\big) \test\, dx.
\end{align*}
Continuity of the forms follow easily from the fact that $n,p,\test\in L^4(\Omega)$ and Lemma~\ref{lem:nemytskii}.

To show coercivity of the bilinear form, one uses Jensen's inequality to obtain
\begin{align*}
 a(\xi,\xi) &= \lambda^2\int_\Omega |\nabla \xi|^2dx + \int_\Omega (n+p) \xi^2 dx - N \left(\int_\Omega \frac{n}{N}\,\xi\, dx\right)^2 - P\left(\int_\Omega \frac{p}{P}\,\xi\,dx\right)^2 \\
 &\ge \lambda^2\int_\Omega |\nabla \xi|^2dx + \int_\Omega (n+p) \xi^2 dx - \int_\Omega n\,\xi^2dx - \int_\Omega p\,\xi^2dx \ge \lambda\|\nabla\xi\|_2^2.
\end{align*}
Since the Poincar\'e inequality $\norm{\xi}_{2} \le c_p \norm{\nabla\xi}_{2}$ holds in $\Y$, we have that the seminorm $\|\nabla \cdot\|_2$ is equivalent to the standard $H^1$-norm. Therefore, $a$ is coercive and Lax-Milgram's theorem yields the required well-posedness result for the adjoint problem.
\end{proof}

\begin{rem}
In the case $\lambda=0$, the adjoint equation becomes a Fredholm type integral equation 
\begin{align}\label{eq:adjoint_zero}
 K[\xi] = f,
\end{align}
where $f$ contains all terms of \eqref{eq:adjoint} independent of $\xi$. As mentioned in Remark~\ref{rem:self-adjoint}, the linear operator $K$ is self-adjoint with a kernel containing all constants. Since
\begin{equation*}
 \langle f,1 \rangle = \int_\Omega K_n[n-n_d]\,dx + \int_\Omega K_p[p-p_d]\,dx = 0,
\end{equation*}
we obtain the unique solvability of the adjoint equation for $\lambda = 0$, which follows directly from the Fredholm alternative theorem\cite{Alt,hackbusch2012integral}.
\end{rem}

Having the well-posedness of the adjoint equations for any $\lambda\ge 0$ at hand, we now state the corresponding first-order optimality system.

\begin{cor}
 The first-order optimality system corresponding to \eqref{optproblem}, $\lambda\ge 0$, reads
 \begin{align*}
  \begin{aligned}
   -\lambda^2\Delta V + B(V) &= 0 \\
   -\lambda^2 \Delta \xi + K[\xi] &= f\\
   \Delta u &= \xi
  \end{aligned}\;
  \begin{aligned}
   \phantom{\lambda^2}\text{in $\Omega$},\\
   \phantom{\lambda^2}\text{in $\Omega$},\\
   \text{in $\Omega$},
  \end{aligned}\quad
  \begin{aligned}
   \phantom{\lambda^2}\nu\cdot \nabla V = 0\\
   \phantom{\lambda^2}\nu\cdot \nabla \xi = 0\\
   \nu\cdot \nabla u = 0
  \end{aligned}\;
  \begin{aligned}
   \phantom{\lambda^2}\text{on $\partial\Omega$},\\
   \phantom{\lambda^2}\text{on $\partial\Omega$},\\
   \text{on $\partial\Omega$},
  \end{aligned}
 \end{align*}
for the tuple $(V,u,\xi)\in\Sigma\times\Y\times\Y$ with $f=K_n[n(V)- n_d] - K_p[p(V)-p_d]\in L^2(\Omega)$.
\end{cor}

\begin{rem}
From the numerical perspective, it is not recommended to solve the first-order optimality system as is. This is due to the nonlocality of the linear operator $K$. Indeed, discretization of a nonlocal operator leads to dense matrices, which would require a large amount of memory. This is especially the case in higher spatial dimensions. Moreover, iterative methods become less efficient for solving such linear systems.

For this reason, we will introduce a path following method for the optimization algorithm. To this end, we define the so-called reduced cost functional $\hat J\colon \Y\to\rr$. Using the control-to-state map $\Phi$ defined in Remark \ref{control-to-state-rem}, the reduced cost functional $\hat{J}$ reads
\[
 \hat{J}(u) := J(\Phi(u+C_{\text{ref}}),u+C_{\text{ref}}).
\]
Using the fact that $\Phi$ is Fr\'echet differentiable (cf.~Lemma~\ref{lem:F_diff}), we differentiate $\hat J$ to obtain
\[
 d\hat J(u)[h] = \langle d_V J(\Phi(u+C_{\text{ref}}),u+C_{\text{ref}}),d\Phi(u+C_{\text{ref}})[h]\rangle + d_uJ(u+C_{\text{ref}})[h],
\]
which is well-defined for any $h\in\Y$. It is easy to see that
\[
 d\Phi(u+C_{\text{ref}})[h] = -d_V e_\lambda(V,u+C_{\text{ref}})^{-1}[d_u e_\lambda(V,u+C_{\text{ref}})][h].
\]
Consequently, for any $h\in\Y$, it follows that
\begin{align*}
 d\hat J(u)[h] &= \langle [d_u e_\lambda(V,u+C_{\text{ref}})]^*\xi + d_uJ(u+C_{\text{ref}}), h\rangle = \int_\Omega \xi\,h\,dx + \int_\Omega \nabla u\cdot \nabla h\,dx,
\end{align*}
where $\xi\in\Y$ satisfies the adjoint equation \eqref{eq:adjoint}. Furthermore, since $\Y$ is a Hilbert space, we may find a corresponding element $g_u \in\Y$ such that $(g_u,h)_{\Y} = d\hat J(u)[h]$. More specifically, $g_u$ satisfies the elliptic equation
\begin{align}\label{eq:gradient}
\int_\Omega \nabla g_u\cdot\nabla\test\,dx = \int_\Omega \xi\,\test \,dx + \int_\Omega \nabla u\cdot \nabla \test \,dx\qquad\text{for all $\test\in\Y$},
\end{align}
which is known to be uniquely solvable in $\Y$.
\end{rem}

\section{$\Gamma$-convergence for the Quasi-neutral Limit $\lambda \rightarrow 0$}\label{sec:gamma}

In this section, we study the $\Gamma$-convergence of the minimization problems \eqref{optproblem}, $\lambda\ge 0$ in the quasi-neutral limit $\lambda \to 0$. An introduction to this topic may be found, e.g., in \cite{Braides,DalMaso}. We will use the following sequential characterization of $\Gamma$-limits that can be found in \cite[p.~86]{DalMaso}:

\begin{prop}[$\Gamma$-convergence of functionals]\label{prop:gamma}
 Let $X$ be a reflexive Banach space with a separable dual and $(F_k)$ be a sequence of functionals from $X$ into $\overline{R}$. Then $(F_k)$ $\Gamma$-converges to $F$ if the following two conditions are satisfied:
 \begin{enumerate}
  \item[(i)] For every $x\in X$ and for every sequence $(x_k)\subset X$ weakly converging to $x$:
  \begin{align}\tag{L-inf}
   F(x) \le \liminf\nolimits_{k \to\infty} F_k(x_k)
  \end{align}
  \item[(ii)] For every $x\in X$ there exists a sequence $(x_k)\subset X$ weakly converging to $x$:
  \begin{align}\tag{L-sup}
    F(x) \ge \limsup\nolimits_{k \to\infty} F_k(x_k)
  \end{align}
 \end{enumerate}
\end{prop}
 
We will also make use of {\em weak equi-coercivity} for functionals on Banach spaces.
\begin{defi}[Equi-coercivity]
 Let $X$ be a reflexive Banach space with separable dual. A sequence of functionals $(F_k)$ is said to be weakly equi-coercive on X, if for every $t \in \rr$ there exists a weakly compact subset $K_t$ of $X$ such that $\{F_k \le t\} \subset K_t$ for every $k \in \nn$.  
\end{defi}
 
Together, these notions lead to the following fact \cite{Braides}:
\begin{center}
 $\Gamma$-convergence $+$ equi-coercivity $\;\Longrightarrow\;$ convergence of minima and minimizers.
\end{center}
This concept of convergence is now applied to our quasi-neutral limit problem. For our specific case, we choose the space $X := H^1(\Omega)\times H^1(\Omega)$, endowed with its weak topology. Notice that $X$, as a product of reflexive Hilbert spaces with separable dual admits the same properties.

To consider the sequence of functionals in the same space, we include the constraint \eqref{eq:weakNPE} into the functional $J$ with the help of a characteristic function. Let $\Sigma$, $\U$ and $J$ be as in Section~\ref{sec:opt}. We define a set of admissible pairs
\[
 \Pi_\lambda := \left\{(V,C) \in \Sigma \times \U\;|\; e_\lambda (V,C) = 0\;\;\text{in}\;H^1(\Omega)^* \right\} 
\]
and its characteristic function 
\[
 \chi_{\Pi_\lambda} = \begin{cases}
  0, & \text{if\, $(V,C)\in \Pi_\lambda$} \\ +\infty, & \text{otherwise}
 \end{cases}.
\]
We now use $\chi_{\Pi_\lambda}$ to include the state equation in the cost functional
\[
 \J_\lambda\colon X\to \rr\cup\{+\infty\};\qquad \J_\lambda = J + \chi_{\Pi_\lambda},
\]
where $J$ is the cost functional given as above.

In the following we consider the extended minimization problem:
\begin{align}\tag{{\bf exOP}$_\lambda$}\label{min mit lambda}
 \text{Find\, $(V_\lambda,C_\lambda)\in X$}:\quad (V_\lambda,C_\lambda) = \arg\min \J_\lambda(V,C). 
\end{align}
In particular, we investigate the behavior of the sequence of minimizers $\{(V_\lambda,C_\lambda)\}$ as $\lambda$ tends to zero. Obviously, a pair $(V_\lambda,C_\lambda)$ solving \eqref{min mit lambda} also solves \eqref{optproblem}.

The first step in the proof of $\Gamma\text{-lim}_{\lambda\to 0}\J_\lambda=\J_0$, is to show that $\Gamma\text{-lim}_{\lambda\to 0} \chi_{\Pi_\lambda}=\chi_{\Pi_0}$ as $\lambda \rightarrow 0$. This is given in the following result.

\begin{lem}
 Let $\Pi_{\lambda_k} $ be defined as above with $\lambda_k \rightarrow 0$ as $k \rightarrow \infty$. Then 
 \[
  \chi_{\Pi_{\lambda_k}}\quad\Gamma\text{-converges to}\quad \chi_{\Pi_0}\quad\text{as\, $k\to\infty$}.
 \]
\end{lem}
\begin{proof}
 We make use of Proposition~\ref{prop:gamma} for the proof.
\begin{itemize}
  \item[(i)] Let $(V,C) \in \Pi_0 $, then $\chi_{\Pi_0}(V,C) \equiv 0$. By definition of $\chi_{\Pi_0}$, the $\liminf$-inequality is satisfied trivially. Now let $z = (V,C) \notin \Pi_0$, then $\chi_{\Pi_0}(z) = \infty$. Let $(z_k)$ be a sequence that converges weakly to $z$ in $\Sigma\times\U$. We assume that $\liminf_{k \rightarrow \infty} \chi_{\Pi_{\lambda_k}} (z_k) = 0$. Then, there exists a subsequence (not relabeled) with $(z_k) \subset \Pi_{\lambda_k}$. 
  In this case, the boundedness of the weakly convergent sequence $(V_k)$ in $H^1(\Omega)$ allows us to pass to the limit in the term
  \[
   \hspace*{4em} \lambda_k^2\int_\Omega \nabla V_k \cdot \nabla \varphi\, dx\;\longrightarrow \; 0\quad\text{for any\, $\test\in H^1(\Omega)$}.
  \]
  Furthermore, the boundedness of $(C_k)$ in $H^1(\Omega)$ provides the boundedness of the sequences $(n(V_k))$, $(p(V_k))$ in $L^1(\Omega)$ and consequently the strong convergence of the (sub)sequences 
  \[
   \hspace*{4em} n(V_k)\to n(V),\qquad p(V_k)\to p(V)\quad\text{in}\quad L^1(\Omega),
  \]
  by the Lebesgue dominated convergence theorem. Hence, the limit $k\to\infty$ gives
  \[
   \hspace*{4em} 0 = \langle e_{\lambda_k}(z_k),\test\rangle \;\;\longrightarrow\;\; \langle e_0(z),\test\rangle\quad\text{for any\, $\test\in \Y$},
  \]
  i.e., $z\in \Pi_0$, which contradicts our assumption $z \notin \Pi_0 $, thereby proving (L-inf).
  
  \item[(ii)] Let $z \notin \Pi_0 $, then $\chi_{\Pi_0} (z) = \infty$. In this case, the $\limsup$-inequality is satisfied trivially. Now let $z \in \Pi_0 $, i.e., $\chi_{\Pi_0}(z) \equiv 0$. Then we find a sequence $(V_k)$ given by Lemma~\ref{lem:main_convergence} that converges weakly towards $V_0$ in $H^1(\Omega)$. Hence, the pair $(V_k,C)=:z_k \rightharpoonup z$ in $X$, satisfies $\chi_{\Pi_{\lambda_k}}(z_k)=0$ for all $k\in\nn$.
 \end{itemize}
 Together, we obtain the $\Gamma$-convergence of the sequence of characteristic functions.
\end{proof}

As a direct result of the previous lemma, we obtain 
\begin{thm}
 Let $\J_{\lambda_k}$ for $\lambda_k \ge 0$ as $k\to\infty$, be defined as above. Then 
 \[
  \J_{\lambda_k}\quad\Gamma\text{-converges to}\quad \J_0\quad\text{as\, $k\to\infty$}.
 \]
\end{thm}
\begin{proof}
As in the previous lemma, we make use of Proposition~\ref{prop:gamma}.
\begin{itemize}
 \item[(i)] Since $\chi_{\Pi_{\lambda_k}}$ satisfies the $\liminf$-inequality, we can now exploit the weak lower semicontinuity of the functional $J$. Let $z = (V,C) \in X$ and $(z_k)$ be a sequence weakly converging to $z$ in $X$. We estimate 
 \begin{align*}
  \hspace*{4em} \J_0(z) = J(z) + \chi_{\Pi_0}(z) &\le \liminf\nolimits_{k \to\infty} J(z_k) + \liminf\nolimits_{k \to\infty} \chi_{\Pi_{\lambda_k}}(z_k) \\
  &\le  \liminf\nolimits_{k \to\infty} \J_{\lambda_k}(z_k),
 \end{align*}
 which is the required inequality.
 
 \item[(ii)] Let $z = (V,C) \in X$. We begin by assuming that $\J_0(z) = \infty$, and define $z_k \equiv z$ for all $k\in \mathbb{N}$. Then for $k_0$ sufficiently large, we have that $\J_{\lambda_k}(z_k) = \infty$ for $k>k_0$, since otherwise the $\liminf$-inequality would be violated. Thus, the $\limsup$-inequality holds. 

 Next, we assume $\J_0(z) < \infty$. In particular, $z \in \Pi_0$. Therefore, Lemma~\ref{lem:main_convergence} ensures the existence of a sequence $(z_k)=(V_k,C)\in \Pi_{\lambda_k}$ converging to $z$ in $X$. For this sequence (or a subsequence thereof), we further have that
 \[
  \hspace*{4em} n(V_k)\to n(V),\qquad p(V_k)\to p(V)\quad\text{in}\quad L^2(\Omega).
 \]
 Therefore, we have $J(z_h) \rightarrow J(z)$ as $h \rightarrow \infty$ and since $z_k \in \Pi_{\lambda_k}$, we obtain 
 \[
  \hspace{4em} \J_0(z) = J(z) + \chi_{\Pi_0}(z) \ge \lim\nolimits_{k\to\infty} J(z_k) + \lim\nolimits_{k\to\infty}\chi_{\Pi_{\lambda_k}}(z) = \limsup\nolimits_{k\to\infty}, \J_{\lambda_k}(z_k),
 \]
 which is the $\limsup$-inequality.
\end{itemize}
With this we conclude our proof.
\end{proof}

\begin{rem}
This is the desired convergence that approves the approximation of the solution to the NNPE \eqref{NPE} with small $\lambda$ by the zero space charge solution \eqref{zero} in optimal semiconductor design. This underlines the assumptions made in \cite{FWIK92semicond1} and allows for the construction of faster optimal design algorithms based on the reduced, algebraic state equation.
\end{rem}

\subsection{Convergence of minima and minimizers}\label{ConvergenceMinimaMinimizer}
In the following, we prove, additionally, the convergence of minima and minimizers.
To obtain this, we have to show equi-coercivity of the functionals $\J_\lambda$. 
\begin{thm}
 Let $(\lambda_k)$ be a sequence with $\lambda_k \rightarrow 0$ as $k \rightarrow \infty$. Then the family of functionals $(\J_{\lambda_k})_k$ is equi-coercive in $X$.
\end{thm}
\begin{proof}

We have to show that $\{ \J_{\lambda_k} < t \}$ is bounded in $X$ for each $t \in \rr$.
 
Let $t < \infty$, then every $(V,C) \in K_t:=\{\J_{\lambda_k}(V,C) \le t \}$ must be in the set of admissible states $\Pi_{\lambda_k}$ for some $\lambda_k>0$, simply due to the presence of the characteristic function. Furthermore, $\|C\|_{H^1}\le t$ due to the radial unboundedness of $J$ w.r.t.~$C$. As seen in the proof of Lemma~\ref{lem:main_convergence}, we have the uniform boundedness of $V$ and therefore the equi-coercivity of the sequence $\J_{\lambda_k}$. 
\end{proof}

Now, the convergence of minima and minimizers under the zero space charge limit $\lambda\to 0$ may be obtained from standard results of $\Gamma$-convergence \cite{DalMaso}.
\begin{cor}
 Let $\J_{\lambda}$, $\J_0$ be defined as above and $(\lambda_k)$ be a sequence with $\lambda_k \rightarrow 0$ as $k \rightarrow \infty$.
 \begin{enumerate}
  \item[(i)] (Convergence of minima) Then $\J_0$ attains its minimum on $X$ with
  \[
   \hspace*{4em} \min\nolimits_{z\in X} \J_0(z) = \lim\nolimits_{k \rightarrow \infty} \inf\nolimits_{z \in X} \J_{\lambda_k}(z).
  \]
  \item[(ii)] (Convergence of minimizers) If $z_k=\arg\min \J_{\lambda_k}$, then there is a subsequence (not relabeled) such that the following holds:
  \[
   \hspace*{4em} z_k  \rightharpoonup z_0\quad\text{in\, $X$},\qquad z_0=\arg\min \J_0 \in \Pi_0.
  \]
 \end{enumerate}
\end{cor}
\begin{proof}
 The proof of (i) follows from \cite[Theorem~7.8]{DalMaso}, while the proof of (ii) follows from \cite[Corollary 7.20]{DalMaso}.
\end{proof}

\section{Numerical Implementation and Results}\label{numerics}
The aim of this section is the numerical illustration of the previous results. At first, solvers for the forward and adjoint problem are proposed and their convergence is shown. Then, the proposed methods are used to solve the optimization problems and compare the results for small $\lambda$ to the zero space charge solution. 

We recall the optimal control problem: Find $(V_*,C_*)\in \Sigma\times \U$ such that
\begin{align}\tag{$\text{\bf OP}_\lambda$}
 (V_*,C_*) = \arg\min J(V,C) \quad \text{subject to}\quad e_\lambda(V,C) = 0\quad\text{in}\quad \Y^*, 
\end{align}
 where the functional is given by 
\[
 J(V,C) = \frac{1}{2} \norm{n(V) - n_d}_{2}^2 + \frac{1}{2} \norm{p(V)- p_d}_{2}^2 + \frac{\sigma}{2} \norm{\nabla (C - C_{\text{ref}})}_{2}^2,
\]
for some given $n_d,p_d\in L^2(\Omega)$, $C_{\text{ref}}\in H^1(\Omega)\cap L^\infty(\Omega)$ and constant $\sigma>0$. Throughout this section we assume $U_{ad} \subset L^\infty(\Omega)$.

The algorithms for the equations are based on the Finite Element Method (FEM) and realized with help of the \texttt{FEniCS} package \cite{logg2012automated} in \texttt{python}. The integral constraints are included in the variational formulations using Lagrange multipliers. Hence, the discrete functions are defined on the mixed function space of linear Lagrange basis functions and $\rr$.

\subsection{Algorithm for the state equation}
The forward solver uses a fixed point iteration in order to avoid the discretization of the nonlocal terms. Keeping the nonlocal terms fixed we solve a local nonlinear monotone equation using Newton's method, then the integral terms are updated. More specifically, we begin by choosing $\alpha,\beta\ge 0$ appropriately. 

We then solve the nonlinear auxiliary problem: Find $V\in \Sigma$ satisfying

\begin{subequations}\label{aux_NNPE}
\begin{align}\label{eq:aux_NP}
 -\lambda^2\Delta V -\alpha e^{-V} + \beta e^{V} + C = 0\quad\text{in $\Omega$},\qquad \nu\cdot\nabla V=0\quad\text{on $\partial\Omega$},
\end{align}
using a Newton iteration procedure. Finally, we update $\alpha,\beta$ according to
\begin{align}\label{eq:update}
 \alpha = \frac{N}{\int_\Omega e^{-V}\,dx},\qquad \beta =\frac{P}{\int_\Omega e^V\,dx}.
\end{align}
\end{subequations}

This procedure continues until convergence is achieved. A pseudocode of the forward solver may be found in Algorithm~\ref{ForwardAlgLam}. The following theorem assures the convergence of the algorithm using monotone methods, see for example \cite{sattinger1972monotone}.

\RestyleAlgo{boxruled}
\LinesNumbered
\begin{algorithm2e}
\caption{Forward Solver ($\lambda > 0$)}\label{ForwardAlgLam}
 \KwData{Initial values $(\alpha_0, \beta_0)\ge 0$, doping profile $C$, stopping tolerance $\tt tol$}
 \KwResult{$(V,n,p)$ satisfying \eqref{NPE}, $\lambda>0$}
 {\bf Initialization}\;
 \While{$\norm{V_k - V_{k-1}}_{H^1(\Omega)} > \tt tol$}{
  \begin{itemize}
  \item [(a)] solve \eqref{eq:aux_NP} for $V_k$ using the Newton's method
  \item [(b)] compute $(\alpha_{k+1},\beta_{k+1})$ according to \eqref{eq:a_update}
  \end{itemize}
 }
\end{algorithm2e}

\begin{thm}
 The sequence $(V_k,\alpha_k,\beta_k)$ constructed by Algorithm~\ref{ForwardAlgLam} converges in $H^1(\Omega)\times \rr^2$ to the unique solution $(V,\alpha,\beta)$ of \eqref{aux_NNPE}.
\end{thm}
\begin{proof}
The idea is to make use of a monotone convergence argument. We therefore construct a bounded sequence $(\alpha_k,\beta_k)$ defined by
\begin{align}\label{eq:a_update}
 \alpha_k = \frac{N}{\int e^{-V_k} }, \quad \beta_k = \frac{P}{\int e^{V_k} },
\end{align}
where $(V_k)$ is obtained in the following. First, we define the nonlinear operator 
 \begin{align}\label{Lop}
  L_v(V) = -\lambda^2\Delta V - N \frac{e^{-V}}{\int e^{-v}} + P \frac{e^V}{\int e^{v}}, 
 \end{align}
 for some $v\in\Sigma$. This operator is known to be strictly monotonically increasing in $V\in\Sigma$.
 
 Now let $\overline{V}$ be the unique solution of the nonlinear, nonlocal problem
 \begin{equation}\label{eq:upper}
  L_{\overline{V}}(\overline{V}) = -\underline{C} = -\min_x C(x)\ge -C,
 \end{equation}
 which is known to exists due to Theorem \ref{thm:main_state}. By construction $\overline{V}$ is a supersolution for the equation $L_{\overline{V}}(V) = -C$, i.e., it satisfies
\begin{equation*}
 L_{\overline{V}}(\overline{V}) \ge -C.
\end{equation*}
Analogously we find a subsolution $\underline{V}$ for the equation $L_{\underline{V}}(V)=-C$, by solving 
\begin{equation}\label{eq:lower}
  L_{\underline{V}}(\underline{V}) = -\overline{C} = -\max_x C(x) \le -C.
\end{equation}
With help of these sub- and supersolutions, we define the intervals
\begin{equation*}
 A = \left[ \frac{N}{\int_\Omega e^{-\underline{V}}\,dx}, \frac{N}{\int_\Omega e^{-\overline{V}}\,dx} \right] =: [\underline{\alpha},\overline{\alpha}], \qquad B = \left[ \frac{P}{\int_\Omega e^{\overline{V}}\,dx}, \frac{P}{\int_\Omega e^{\underline{V}}\,dx}\right] = [\underline{\beta},\overline{\beta}].
\end{equation*}
The task is now to show that the sequence $(\alpha_k,\beta_k)$ remain within the interval $A\times B$.

Let $V_0$ be the unique solution of 
\begin{equation}\label{eq:V0}
  L_{\overline{V}}(V_0) = -C.
 \end{equation}
 The following calculations show that $V_0 \le \overline{V}$ a.e.~in $\Omega$. Indeed, by subtracting \eqref{eq:upper} from \eqref{eq:V0} and testing with $(V_0-\overline{V})^+=\max\{0, V_0-\overline V\}$, we obtain
 \begin{align*}
  \int_\Omega |\nabla (V_0-\overline{V})^+|^2dx \le 0,
 \end{align*}
 where we used the fact that $(e^{-x}-e^{-y})(x-y)\le 0$ and $(e^{x}-e^{y})(x-y)\ge 0$ for any $x,y\in\rr$. Since $\|\nabla\cdot\|_2$ is a norm in $\Y$, we conclude that $\text{meas}(\{ x\in\Omega\,|\, V_0 >\overline V\}) = 0$. Analogously one obtains the lower bound $\underline{V} \le V_0$ by subtracting \eqref{eq:V0} from \eqref{eq:lower} and testing with $(\underline V - V_0)^+$. 
 
 For a sub- and supersolution $\underline V_0, \overline V_0$ of \eqref{eq:V0} respectively, we define
 \[
  \overline\alpha_0 = \frac{N}{\int_\Omega e^{-\overline{V}_0}\,dx},\qquad \underline\alpha_0 = \frac{N}{\int_\Omega e^{-\underline{V}_0}\,dx},\qquad \underline\beta_0 = \frac{P}{\int_\Omega e^{\overline{V}_0}\,dx},\qquad \overline\beta_0 = \frac{P}{\int_\Omega e^{\underline{V}_0}\,dx}.
 \]
 Clearly, we have the chain of inequalities
 \[
  \underline V\le \underline V_0\le V_0\le \overline V_0\le \overline V\quad\text{a.e.~in $\Omega$}.
 \]
 Due to the monotonicity of $e^{-x}$ and $e^x$, we further obtain
 \[
  \underline \alpha \le \underline \alpha_0 \le \alpha_0 \le \overline \alpha_0 \le  \overline\alpha,\qquad \underline \beta \le \underline \beta_0 \le \beta_0 \le \overline \beta_0 \le  \overline\beta.
 \]
 The iteration proceeds by solving 
 \[
  L_{\overline{V}_{k-1}}(V_k) = -C,
 \]
 resulting in the sequence $(\alpha_k,\beta_k)$ within the compact interval $A\times B$. Consequently, the constructed sequence $(\alpha_k,\beta_k)$ admits accumulation points $(\alpha_*,\beta_*)$ such that
 \begin{equation*}
  \alpha_k \longrightarrow \alpha_*, \qquad \beta_k \longrightarrow \beta_* \quad\text{for $k \rightarrow \infty$},
 \end{equation*}
 hold for subsequences (not relabeled) of $(\alpha_k,\beta_k)$. 
 
 The corresponding solution $V^*$ of the auxiliary problem \eqref{eq:aux_NP} solves \eqref{NPE}. Since, the solution of \eqref{NPE} is known to be unique, the complete sequence $(V_k)$ converges towards $V_*\in\Sigma$.
\end{proof}

In the case $\lambda = 0$, the equation to solve for $V$ is a nonlinear algebraic equation instead of a nonlinear Poisson equation. In fact, the proof using sub- and supersolutions may be directly transfered to this case. As shown in Lemma~\ref{lem:apriori_zero}, the solution satisfies a given equation. As before, we perform the iteration for $(\alpha,\beta)$ in this case as well, see Algorithm~\ref{ForwardAlg}.

\begin{algorithm2e}
\caption{Forward Solver ($\lambda = 0$)}\label{ForwardAlg}
 \KwData{Initial values $(\alpha_0,\beta_0)\ge 0$, doping profile $C$, stopping tolerance $\tt tol$}
 \KwResult{$(V,n,p)$ satisfying \eqref{zero}}
 {\bf Initialization}\;
 \While{$\norm{V_k - V_{k-1}}_{2} > \tt tol$ }{
  \begin{itemize}
  \item [(a)] compute $V_k$ using \eqref{VinL0}
  \item [(b)] compute $(\alpha_{k+1},\beta_{k+1})$ according to \eqref{eq:a_update}
  \end{itemize}
 }
\end{algorithm2e}

\subsection{Algorithm for the adjoint equation}

As mentioned in Section~\ref{sec:opt}, a direct discretization of the nonlocal terms results in dense matrices. In order to avoid this, we define two new variables, namely

\begin{subequations}\label{adjoint_new}
\begin{equation}\label{eq:Lagrange-Multipliers}
 \xi^\alpha = \frac{1}{N}\int_\Omega n\, [\xi  - (n-n_d)]\,dx, \qquad \xi^\beta = \frac{1}{P}\int_\Omega p\, [\xi +  (p-p_d)]\,dx.
\end{equation}
Consequently, the adjoint equation \eqref{eq:adjoint} may be equivalently written as
\begin{align}\label{eq:Adjoint-Poisson}
 -\lambda^2 \Delta \xi + (n + p) \xi - n\,\xi^\alpha - p\,\xi^\beta &= n(n-n_d) - p(p-p_d)\qquad \text{in}\;\;\Omega. 
\end{align}
\end{subequations}
Note that the occurring matrices are sparse. This formulation is used to define Algorithm~\ref{AdjointAlgLam} for the adjoint problem in the case $\lambda >0 $.

We begin the iteration by solving the linear, local equation \eqref{eq:Adjoint-Poisson} for $\xi$ using the fixed values of the parameters $\xi^\alpha$ and $\xi^\beta$. Then the parameters are updated and the new Poisson equation is solved. This procedure continues until convergence is achieved, see Algorithm \ref{AdjointAlgLam}. The convergence the the algorithm is proven by the following theorem.

\begin{algorithm2e}
\caption{Adjoint Solver ($\lambda > 0$) }\label{AdjointAlgLam}
 \KwData{Initial values $(\xi^\alpha_0,\xi^\beta_0)$, solution the forward problem $V$, stopping tolerance ${\tt tol}$}
 \KwResult{$(\xi,\xi^\alpha,\xi^\beta)$ solving the adjont system \eqref{adjoint_new}}
 {\bf Initialization}\;
 \While{$\norm{\xi_k - \xi_{k-1}}_{H^1(\Omega)} > {\tt tol}$ }{
  \begin{itemize}
  \item [(a)] solve the adjoint problem \eqref{eq:Adjoint-Poisson} for $\xi_k$ with fixed $(\xi_k^\alpha$, $\xi_k^\beta)$
  \item [(b)] update $\xi_k$ to satisfy $\int \xi_k\, dx =0$
  \item [(c)] compute $(\xi^\alpha_{k+1}, \xi^\beta_{k+1})$ according to \eqref{eq:Lagrange-Multipliers}
  \end{itemize}
 }
\end{algorithm2e}

\begin{thm}\label{thm:conv_adjoint}
Let $V\in L^\infty(\Omega)$ be given. Then, the sequence $(\xi_k,\xi^\alpha_k,\xi^\beta_k)$ constructed by Algorithm \ref{AdjointAlgLam} converges in $H^1(\Omega)\times \rr^2$ to the solution $(\xi,\xi_\alpha,\xi_\beta)$  of \eqref{adjoint_new} for $\lambda > 0$.
\end{thm}
\begin{proof}
 Note that regarding the adjoint system the values of $V, n, p, n_d$ and $p_d$ are known. Denoting the difference of two consecutive iterates as
 \[
  \varepsilon^\alpha_k = \xi^\alpha_k-\xi^\alpha_{k-1},\qquad \varepsilon^\beta_k = \xi^\beta_k-\xi^\beta_{k-1},\qquad \varepsilon_k = \xi_k-\xi_{k-1},
 \]
 we apply the H\"older inequality to obtain the estimates
 \[
  N|\varepsilon^\alpha_k|^2 \le \int_\Omega n\,\varepsilon_{k-1}^2\,dx,\qquad P|\varepsilon^\beta_k|^2 \le \int_\Omega p\,\varepsilon_{k-1}^2\,dx
 \]
 Similarly, we subtract \eqref{eq:Adjoint-Poisson} for two consecutive iterates of $\xi$ and testing with $\varepsilon_k$ result in
 \[
  \int_\Omega \lambda^2| \nabla \varepsilon_k|^2 + (n+p)\, \varepsilon_k^2\, dx = \int_\Omega \big(n\,\varepsilon^\alpha_{k} + p\,\varepsilon^\beta_{k}\big) \varepsilon_k\, dx.
 \]
 Using Young's inequality we estimate as follows
 \[
  \int_\Omega \lambda^2| \nabla \varepsilon_k|^2 + (n+p)\, \varepsilon_k^2\, dx \le \frac{1}{2}   \left( N\,|\varepsilon^\alpha_{k}|^2 + \int_\Omega(n+p)\,\varepsilon_k^2\,dx + P\,|\varepsilon^\beta_{k}|^2\,\right).
 \]
With help of the Poincaré inequality we estimate further to obtain
\begin{align*}
 c_0(\lambda) \int_\Omega \varepsilon_k^2\,dx + \int_\Omega (n+p)\,\varepsilon_k^2\,dx \le N |\varepsilon^\alpha_{k}|^2 + P |\varepsilon^\beta_{k}|^2.
\end{align*}
with $c_0(\lambda) = 2\lambda^2/c_p^2$, where $c_p$ is the Poincar\'e constant. Since $n,p \in L^\infty(\Omega)$ (cf.~Lemma~\ref{lem:main_regular}), we find some $M>0$ such that $n + p\le M$. Therefore, we obtain
\[
 \int_\Omega (n+p)\,\varepsilon_k^2\,dx \le \frac{1}{c_M} \left( N |\varepsilon^\alpha_{k}|^2 + P |\varepsilon^\beta_{k}|^2 \right)
\]
with $c_M = 1 +c_0(\lambda)/M > 1$.
Altogether we obtain the contraction
\begin{align*}
 N |\varepsilon^\alpha_{k}|^2 + P |\varepsilon^\beta_{k}|^2 \le \int_\Omega (n+p)\,\varepsilon_{k-1}^2\,dx \le c_M \left( N |\varepsilon^\alpha_{k-1}|^2 + P |\varepsilon^\beta_{k-1}|^2 \right)
\end{align*}
Thus, the algorithm defines a contraction 
\begin{equation*}
   f \colon \mathbb{R}^2 \rightarrow \mathbb{R}^2, \quad f(\xi^\alpha_k,\xi^\beta_k) = (\xi^\alpha_{k+1},\xi^\beta_{k+1})
\end{equation*}
Since $\mathbb{R}^2$ is a complete metric space, the algorithm converges to the unique fixed point, which solves the adjoint problem.
\end{proof}

\begin{rem}
Note that the prefactor determining if the function is a contraction or not, depends on $\lambda^2$ being positive. For small $\lambda$ the contraction property deteriorates. A workaround would be to scale the domain $\Omega$ by $1/\lambda$, which results in a vanishing $\lambda^2$. The drawback is the large domain arising from this scaling. Computational tests have shown that the large domain and therefore huge matrices, which require more memory, have greater negative impact than the $\lambda^2$ in the contraction.
\end{rem}

\begin{rem}
As stated above for $\lambda=0$, the adjoint equation is a Fredholm type integral equation. The discretization is realized again in \texttt{FEniCS} with linear testfunctions resulting in the system
\begin{equation*}\label{eq:IE} \tag{\text{\bf IE}}
 (MD - MH) \overline{\xi} =  F,
\end{equation*}
where $M$ is the mass matrix of the finite elements, $D$ the diagonal matrix with the entries of the discretized $n + p$, $H$ the matrix containing the discretization of the integral kernel, $\overline{\xi}$ denotes the vector of the discretized $\xi$ and $F$ is the discretized version terms on the right hand side of \eqref{eq:adjoint_zero}.
\end{rem}

\begin{rem}
 In fact, the parameters $\xi^\alpha$ and $\xi^\beta$ coincide with the adjoint variables corresponding to $\alpha$ and $\beta$ respectively, if one considers the operator equation for $e_\lambda(V,C)$ as
 \begin{equation*}
  \langle\hat{e}_\lambda(V,\alpha,\beta,C),\test\rangle = 0\qquad\text{for all $\test\in \Y\times \rr\times\rr$}
 \end{equation*}
 where the operator $\hat e_\lambda$ is defined for any $\test=(\xi,\xi^\alpha,\xi^\beta)\in \Y\times\rr\times\rr$ by
 \begin{gather*}
  \langle\hat{e}_{\lambda,1},\xi\rangle = -\lambda^2 \int_\Omega \nabla V\cdot\nabla\xi\,dx - \int_\Omega \big(\alpha e^{-V} - \beta e^V - C\big)\xi\,dx, \\
  \langle\hat{e}_{\lambda,2},\xi^\alpha\rangle = \left(\alpha\int_\Omega e^{-V} dx - N\right)\xi^\alpha,\qquad 
  \langle\hat{e}_{\lambda,3},\xi^\beta\rangle = \left(\beta\int_\Omega e^V dx - P\right)\xi^\beta.
 \end{gather*}
 In this case, the cost functional $J$ will have to be reformulated in terms of $(V,\alpha,\beta)$, i.e.,
 \[
  J(V,\alpha,\beta,C) = \frac{1}{2} \norm{\alpha e^{-V} - n_d}_{2}^2 + \frac{1}{2} \norm{\beta e^V- p_d}_{2}^2 + \frac{\sigma}{2} \norm{\nabla (C - C_{\text{ref}})}_{2}^2.
 \]
 To see that $\xi^\alpha$ is indeed the adjoint variable corresponding to $\alpha$, we formally compute the Gat\'eaux derivative of the Lagrangian
 \[
  \mathcal{L}(V,\alpha,\beta,C,\test) = J(V,\alpha,\beta,C) + \langle\hat{e}_\lambda(V,\alpha,\beta,C),\test\rangle,
 \]
 with respect to $\alpha$ in any direction $h\in\rr$, to obtain
 \begin{align*} 
  d_\alpha \mathcal{L}(V,\alpha,\beta,C,\test)[h] &= \int_\Omega (\alpha e^{-V}-n_d)e^{-V}h\,dx - \int_\Omega  e^{-V}\xi h\,dx + \int_\Omega e^{-V}\xi^\alpha h \,dx =0.
 \end{align*}
 Reformulating the equation in terms of $\xi^\alpha\in\rr$, we obtain
 \[
  \xi^\alpha = \frac{1}{\int e^{-V}dx}\int_\Omega  e^{-V}[\xi - (n-n_d)]\,dx = \frac{1}{N}\int_\Omega n[\xi - (n-n_d)]\,dx,
 \]
 as required, where we used the fact that $n=\alpha e^{-V}$.
\end{rem}

\subsection{Algorithm for the optimal control problem}

The forward and adjoint solvers stated above are used in Algorithm~\ref{OptAlg} for the computation of the optimal control. 
\begin{algorithm2e}
\caption{Optimal Control }\label{OptAlg}
 \KwData{Initial doping profile $C_0$ and potential $V$, and stopping tolerances ${\tt tol}_{opt}$, ${\tt tol}_{abs}$}
 \KwResult{Optimal doping profile, potential, electron and hole densities $(C_*,V_*,n_*,p_*)$}
 {\bf Initialization}\;
 \While{relative norm of gradient $> {\tt tol}_{opt}$ \textbf{and} norm of gradient $> {\tt tol}_{abs}$ }{
  \begin{itemize}
  \item [(a)] solve forward problem with either Algorithm \ref{ForwardAlgLam} or \ref{ForwardAlg}
  \item [(b)] solve adjoint problem with either Algorithm \ref{AdjointAlgLam} or \eqref{eq:IE}
  \item [(c)] solve Riesz representation problem \eqref{eq:riesz} to obtain $g_k$
  \item [(d)] compute $u_{k+1}$ according to \eqref{eq:update_u} using Algorithm~\ref{ArmRule}
  \end{itemize}
 }
\end{algorithm2e}

As pointed out in Section~\ref{sec:opt}, the optimality condition for the optimal control is given as a system of partial differential equations. However, instead of solving the system directly, we consider a path following optimization procedure, which iteratively determines a better approximation of a local minimizer. More specifically, we consider a steepest descent algorithm for the reduced cost functional $\hat J$, in which we update the doping profile according to
\begin{align}\label{eq:update_u}
  u_{k+1} = u_k + \omega g_k,\quad k\ge 0,
\end{align}
where $g_k\in\Y$ denotes the correct gradient representation of the derivative $d\hat J$, obtained uniquely by solving the Riesz representation problem
\begin{align}\label{eq:riesz}
 \int_\Omega  \nabla g_k\cdot \nabla \test \,dx = \int_\Omega \nabla u_k \cdot \nabla \test + \xi_k\, \test \,dx \qquad \text{for all $\varphi \in \Y$}.
\end{align}
To obtain an appropriate stepsize $\omega$ we apply the Armijo stepsize rule (cf.~\cite{Pinnau}).
\begin{algorithm2e}[h!]
\caption{Armijo Stepsize Rule }\label{ArmRule}
 \KwData{Current iterare $u_k$, gradient $g_k$, initial $\omega_0$, initial $\gamma>0$}
 \KwResult{New iterate $u_{k+1}$}
 {\bf Initialization}\;
 \While{ $\hat{J}(u_k + \omega g_k) \ge  \hat{J}(u_k) + \gamma\,\omega \norm{g_k}^2$ }{
   $\omega = \omega/2$
 }
\end{algorithm2e}

\subsection{Numerical influence of $\lambda$}

A comparison of the computation times for the state and adjoint solutions for different $\lambda\ge 0$ can be found in Table \ref{computationalTime}. As seen in Table~\ref{computationalTime}, the computation of the adjoint ($\lambda>0$) is, on average, cheaper than the one of the state equation. For the state problem, the Newton iteration requires several solves of the auxiliary problem \eqref{eq:aux_NP}. Thus, the most computational effort can be salvaged when reducing the state problem to the $\lambda=0$ case. Instead of solving many nonlinear differential equations for the Newton iteration, one iterates between $\alpha,\beta$ and the explicit solution of $V$ for $\lambda=0$ given in \eqref{VinL0}.
\begin{table}[h]\centering
\begin{tabular}{|c|c|c|c|} 
  \hline
  $\lambda^2$ & State [s] & Adjoint [s] & $\norm{C_\lambda - C_0}_{2}$ \\
  \hline
  1e-3 &   2.00 & 2.04 & 0.064737 \\
  1e-4 &   6.61 & 4.91 & 0.056407 \\
  1e-5 &   25.51 & 10.68 & 0.032445 \\
  1e-6 &   69.63 & 11.8 & 0.0014969 \\
  1e-7 & 117.66 & 11.3 &  0.0008388\\
  1e-8 & 146.55 & 12.5 & 0.0007833\\
  1e-9 & 193.36 & 10.5 & 0.0007807\\
  0.0	& 0.04   & 1.13 & 0 \\
  \hline
 \end{tabular}\\[1em]
 \caption{Computation times for the solution of \eqref{NPE} and its adjoint equation for different $\lambda$. Difference between the optimal doping profiles for different $\lambda\ge 0$ and $\lambda=0$ in the $L^2$-norm.}
 \label{computationalTime}
\end{table}

While the computation times for the solution of the adjoint problem are stable for $\lambda>0$, one observes the increase in computational time for small $\lambda>0$. Indeed, for small $\lambda>0$, the corresponding discretization matrices become stiffer and therefore require more iterations when solving linear systems. Using the zero space charge solution instead of the solution for $\lambda^2 = 10^{-9}$, the result of the forward problem is calculated 1000 times faster. Thus, from the computational point of view the approximation of the optimal controls in cases with small $\lambda>0$ with the zero space charge solution is very useful. The fourth column of Table~\ref{computationalTime} shows the $L_2$ difference of the optimal doping profiles for different $\lambda$ to the zero space charge solution. The values underline the convergence proved in Section~\ref{sec:gamma} as well.

\subsection{Numerical results}
In this section, we use Algorithm~\ref{OptAlg} to obtain numerical results for the optimal control problems. The parameters are set to the values given in Table~\ref{Params}. 
\begin{table}[h]\centering
\begin{tabular}{|c|c|c|c|} 
  \hline
  parameter & value & parameter & value \\
  \hline
  domain $\Omega$  & $[0,1]$  & $\gamma$ & 0.0001   \\
  grid points  & 200 & $\omega_0$   &  50  \\
   ${\tt tol}$& 1e-8 & $\sigma$ & 1e-4  \\
   ${\tt tol}_{opt}$ & 5e-2 & $\delta$  & 10e-4  \\
   ${\tt tol}_{abs}$ & 5e-5  & & \\
  \hline
 \end{tabular}\\[1em]
 \caption{Parameter values for the optimization algorithm}
 \label{Params}
\end{table}

The non-symmetric reference doping profile $C_{\text{ref}}$ depicted in Figure~\ref{fig:CrefStart}(left) serves as the initial doping profile for the optimization. Note that the desired electron and hole densities in Figure~\ref{fig:CrefStart}(right) are not attainable due to the constraints 
\[
 \int_\Omega (C - C_\text{ref})\,dx = 0,\qquad N = \delta^2 +\int_\Omega C^+ dx,\qquad P = \delta^2-\int_\Omega C^- dx.
\]
By choosing the reference doping profile as initial doping profile, the first constraint is trivially satisfied. With the desired densities given in Figure~\ref{fig:CrefStart}(right), the aim of the cost functional is to reduce the electron density and to increase the hole density. 

\begin{figure}
\centering
    \begin{minipage}{0.49\linewidth}
        \centering
        \includegraphics[width = 1.\textwidth]{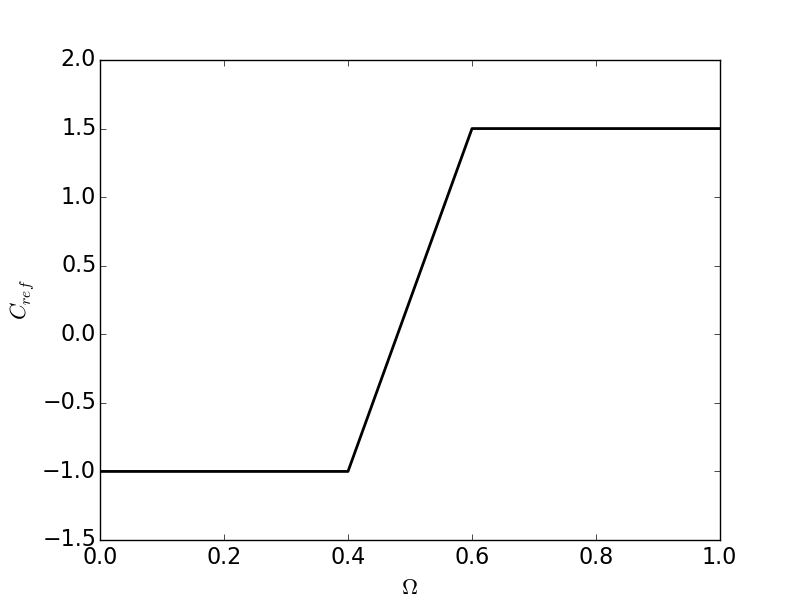}
    \end{minipage}
    \begin{minipage}{0.49\linewidth}
        \centering
        \includegraphics[width = 1.\textwidth]{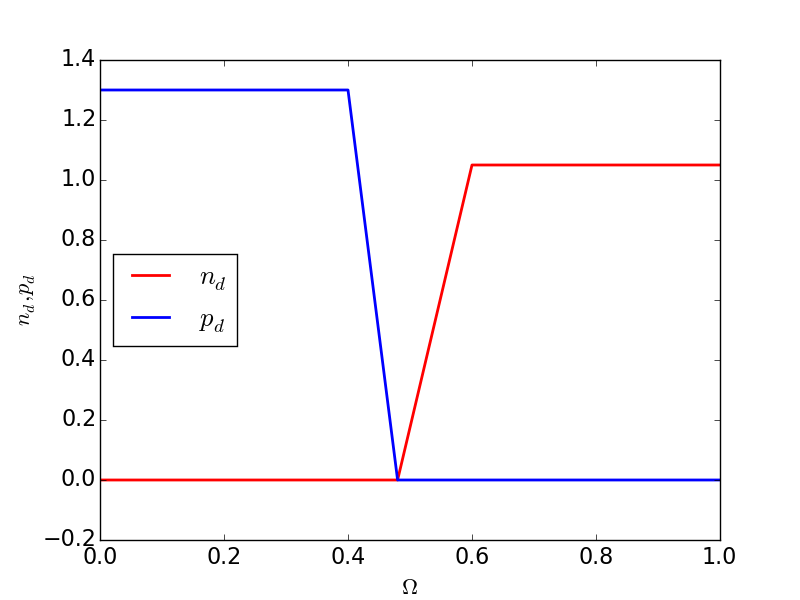}
    \end{minipage}
    \caption{Left: Initial and reference doping profile $C_{\text{ref}}$. Right: Desired electron density $n_{d} = 0.8 C_{\text{ref}}^+$ (red) and hole density $p_{d} = 1.2C_{\text{ref}}^-$ (blue).}\label{fig:CrefStart}
\end{figure}

\begin{figure}
\centering
    \begin{minipage}{0.49\linewidth}
        \centering
        \includegraphics[width = 1.\textwidth]{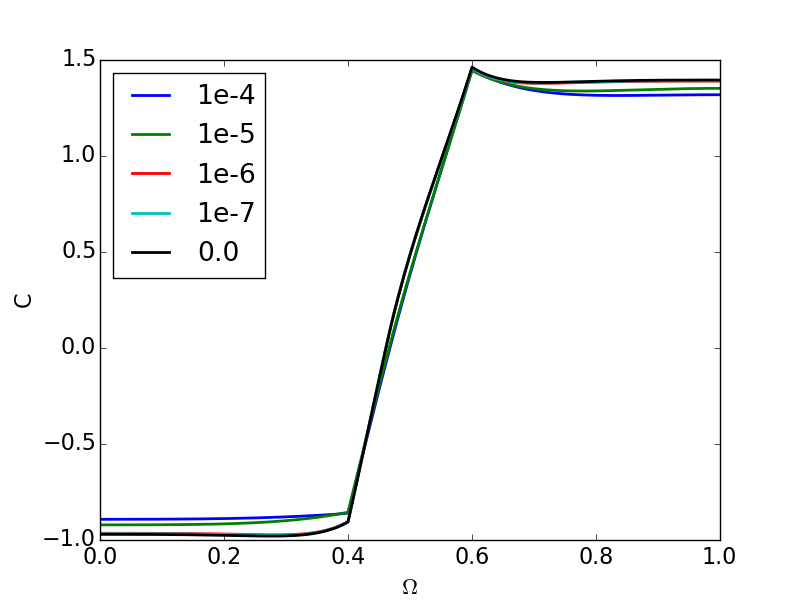}
    \end{minipage}
    \begin{minipage}{0.49\linewidth}
        \centering
        \includegraphics[width = 1.\textwidth]{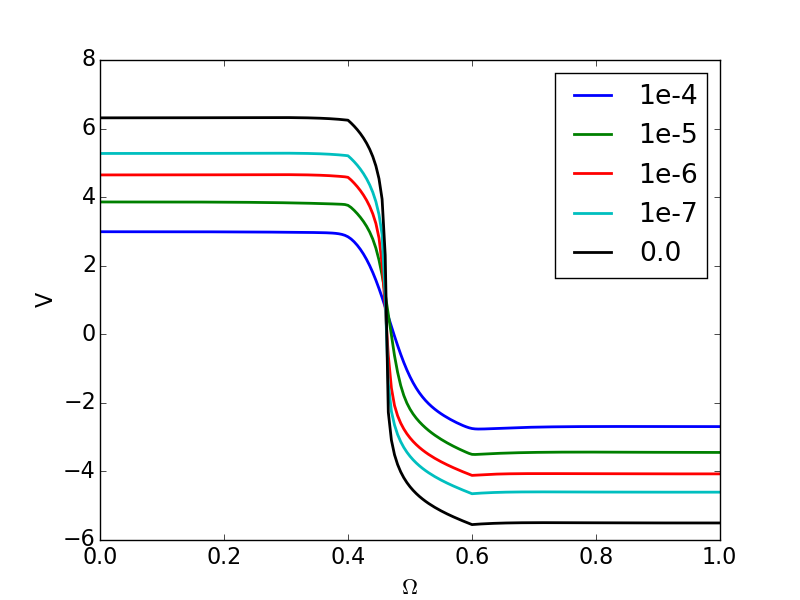}
    \end{minipage}
    \caption{Left: The optimal doping profile for different values of $\lambda^2$. Right: The optimal potential for different values of $\lambda^2$.}
    \label{fig:CVopt}
\end{figure}

\begin{figure}
\centering
    \begin{minipage}{0.49\linewidth}
        \centering
        \includegraphics[width = 1.\textwidth]{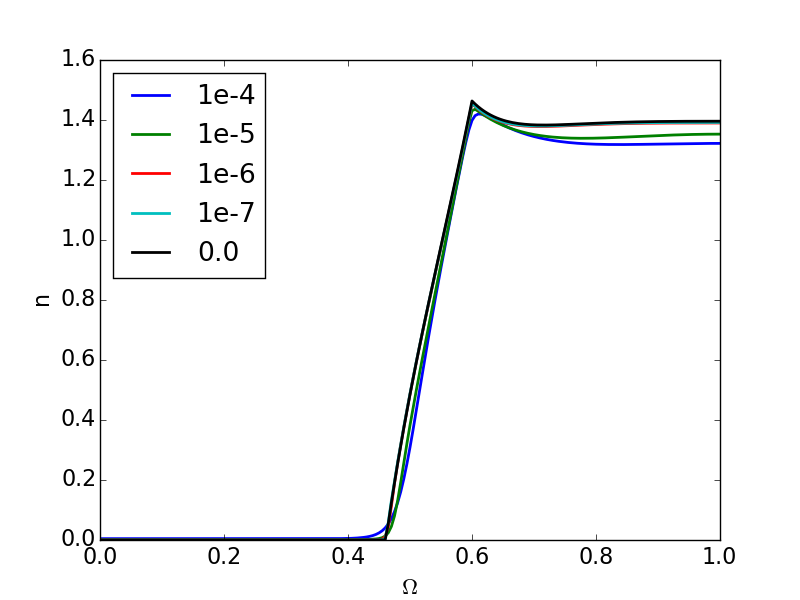}
    \end{minipage}
    \begin{minipage}{0.49\linewidth}
        \centering
        \includegraphics[width = 1.\textwidth]{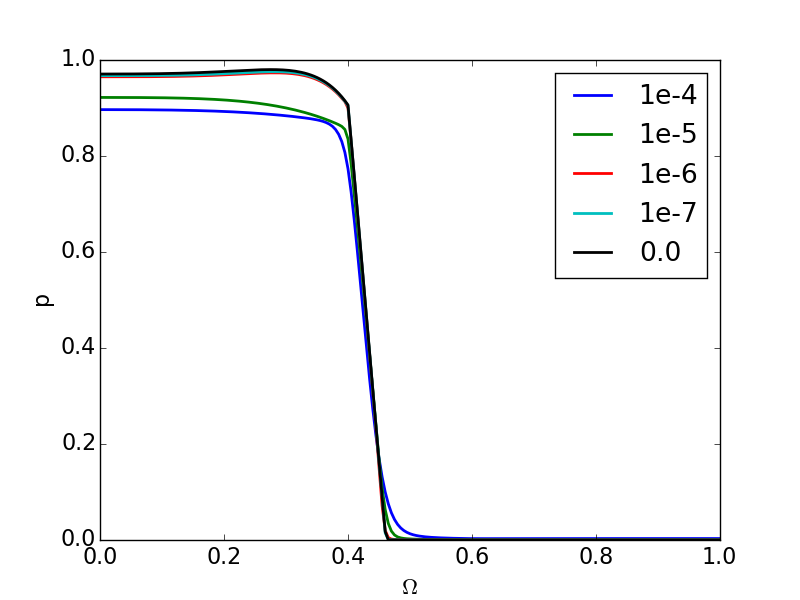}
    \end{minipage}
    \caption{Left: The optimal electron density for different values of $\lambda^2$. Right: The optimal hole density for different values of $\lambda^2$.}
    \label{fig:NPopt}
\end{figure}

In Figure \ref{fig:CVopt}(left) the optimal doping profiles for different $\lambda$ are depicted. As predicted by the theory in Section~\ref{sec:gamma}, the optimal doping profiles converges to the zero space charge optimal doping profile for decreasing values of $\lambda$. Note, that the convergence of the densities depicted in Figure~\ref{fig:NPopt} is more pronounced than the  convergence of the optimal potentials in Figure~\ref{fig:CVopt}(right). This difference in behaviour of the potential $V$ and the densities may be explained by the structure of their expressions, since $V$ appears in the exponential terms of $n$ and $p$. 

In the following plots, we investigate the cost functional and parts of it, in which we set
\[
 J_1(n) = \frac{1}{2}\norm{n(V) - n_d}_{2}^2,\qquad J_2(p) = \frac{1}{2} \norm{p(V)- p_d}_{2}^2,\qquad J_3(C)=\frac{\sigma}{2} \norm{\nabla (C - C_{\text{ref}})}_{2}^2.
\]
In Figure~\ref{fig:J_J1}(left), one clearly sees the monotone decrease in the cost functional $J$, as required. The positive part of the doping profile and analogously the electron density is reduced compared to the starting reference profile. On the other hand, we notice that the hole density is reduced as well, although the goal of $J_2$ is to increase the hole density. This happens because the initial values of $J_1$ are larger than those of $J_2$ (cf. Figure~\ref{fig:J_J1}(right) and Figure~\ref{fig:J2_J3}(left), respectively). The optimization therefore tends to minimize $J_2$ first. Due to the integral constraints on $C$ and the definitions of $N$ and $P$ stated above, it is not possible to enlarge the $p$ part. This behaviour is similar for all values of $\lambda$. 

The simulations with larger $\lambda^2$ ($10^{-4}$ and $10^{-5}$) do one Armijo step at the beginning. In the case $\lambda^2=10^{-4}$, there are two more Armijo steps needed in the second and third iteration. All other iterates accept the new iterate immediately.

\begin{figure}
\centering
    \begin{minipage}{0.49\linewidth}
        \centering
        \includegraphics[width = 1.\textwidth]{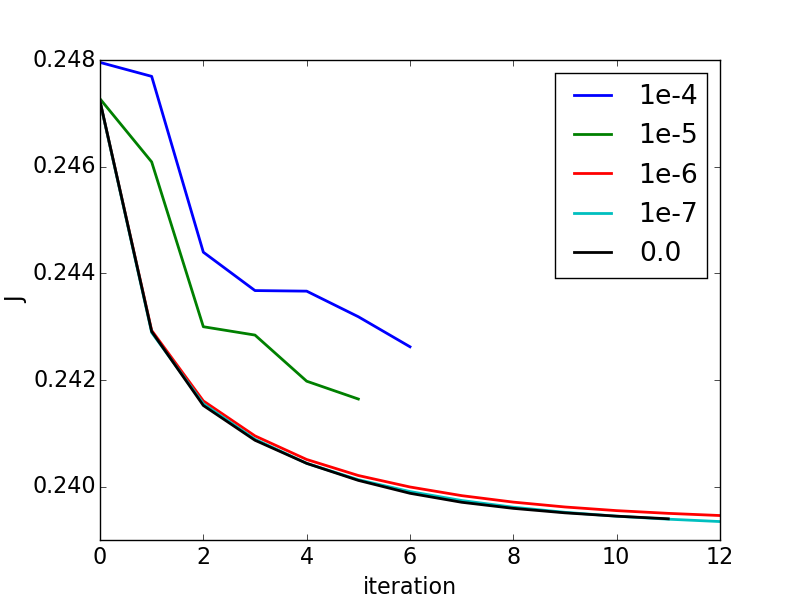}
    \end{minipage}
    \begin{minipage}{0.49\linewidth}
        \centering
        \includegraphics[width = 1.\textwidth]{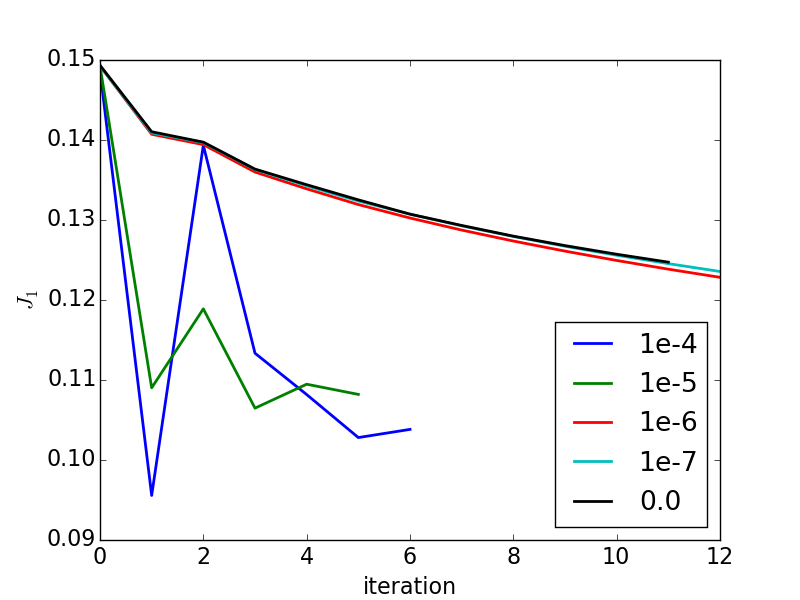}
    \end{minipage}
    \caption{Left: Development of $J$ for different values of $\lambda^2$. Right: Development of $J_1$ for different values of $\lambda^2$.}
    \label{fig:J_J1}
\end{figure}

\begin{figure}
\centering
    \begin{minipage}{0.49\linewidth}
        \centering
        \includegraphics[width = 1.\textwidth]{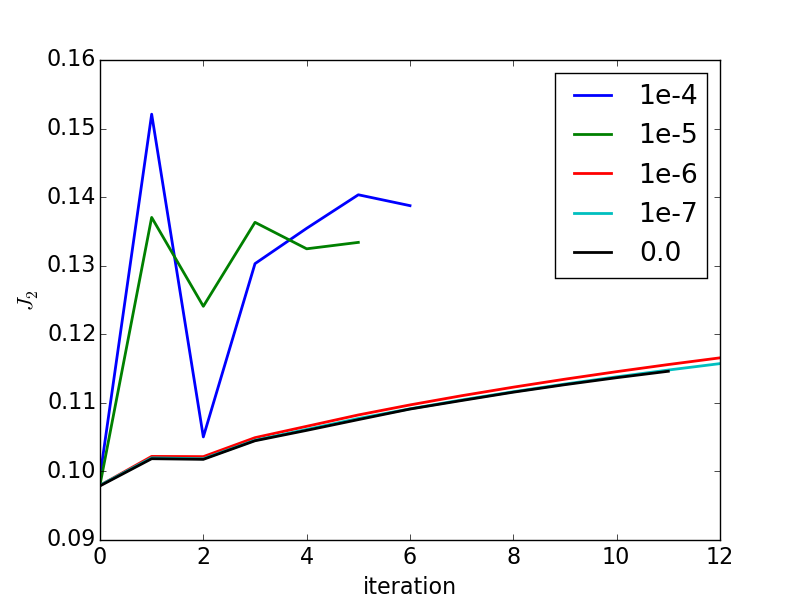}
    \end{minipage}
    \begin{minipage}{0.49\linewidth}
        \centering
        \includegraphics[width = 1.\textwidth]{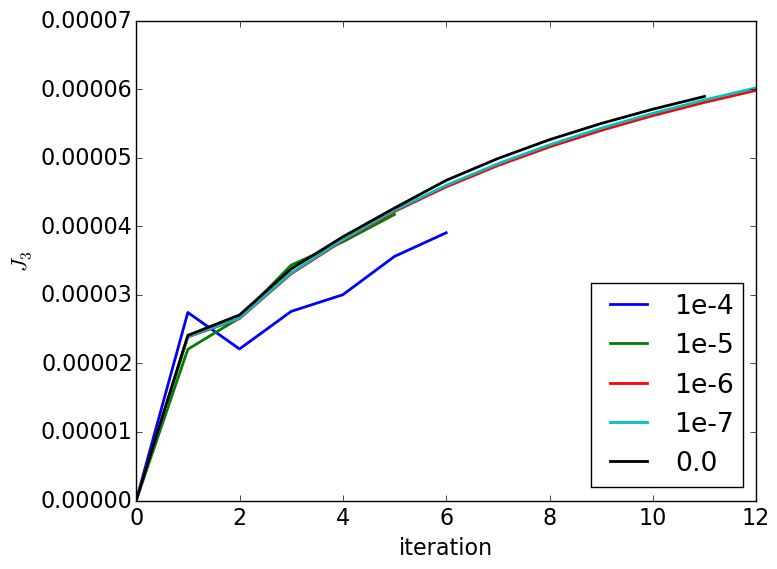}
    \end{minipage}
    \caption{Left: Development of the second term of cost functional for different values of $\lambda^2$. Right: Development of the third term of the cost functional for different values of $\lambda^2$.}
    \label{fig:J2_J3}
\end{figure}



\bibliographystyle{plain}
\bibliography{biblio}

\end{document}